\documentclass[english,12pt]{article}
\usepackage[T1]{fontenc}
\usepackage[utf8]{inputenc}
\usepackage{babel}
\usepackage[numbers]{natbib}
\usepackage{prettyref}
\usepackage{booktabs}
\usepackage{mathtools}
\usepackage{amsmath}
\usepackage{amsthm}
\usepackage{amssymb}
\usepackage{stmaryrd}
\usepackage{graphicx}
\usepackage[unicode=true,
  bookmarks=true,bookmarksnumbered=false,bookmarksopen=false,
breaklinks=false,pdfborder={0 0 1},backref=false,colorlinks=false]
{hyperref}
\hypersetup{pdftitle={Weakly chained matrices, policy iteration, and
  impulse control},
pdfauthor={Parsiad Azimzadeh, Peter A. Forsyth}}

\makeatletter

\providecommand{\tabularnewline}{\\}

\theoremstyle{plain}
\newtheorem{thm}{\protect\theoremname}
\newtheorem{defn}[thm]{\protect\definitionname}
\newtheorem{prop}[thm]{\protect\propositionname}
\newtheorem{rem}[thm]{\protect\remarkname}
\newtheorem{lem}[thm]{\protect\lemmaname}
\ifx\proof\undefined
\newenvironment{proof}[1][\protect\proofname]{\par
  \normalfont\topsep6\p@\@plus6\p@\relax
  \trivlist
  \itemindent\parindent
\item[\hskip\labelsep\scshape #1]\ignorespaces
}{%
  \endtrivlist\@endpefalse
}
\providecommand{\proofname}{Proof}
\fi
\newtheorem{example}[thm]{\protect\examplename}
\newtheorem{cor}[thm]{\protect\corollaryname}

\@ifclassloaded{siamltex}{

  \slugger{sinum}{xxxx}{xx}{x}{x--x}

  \usepackage{caption}
  \captionsetup{labelfont=sc, textfont=it, font=footnotesize}
  \captionsetup[table]{labelsep=newline, skip=2pt}
  \captionsetup[subfloat]{labelfont=normalfont}

  \newtheorem{example}[theorem]{Example}
  \newtheorem{rem}[theorem]{Remark}

  \hypersetup{
    colorlinks,
    linkcolor={black},
    citecolor={black},
    urlcolor={black}
  }

}{

  \pdfoutput=1

  \usepackage{microtype}

  \usepackage[margin=1in]{geometry}

  \usepackage[T1]{fontenc}
  \usepackage{lmodern}

  \newtheorem*{problem*}{Problem}

  \numberwithin{equation}{section}
  \numberwithin{thm}{section}

  \usepackage[dvipsnames]{xcolor}

  \definecolor{mylinkcolor}{HTML}{0066cc}
  \definecolor{mycitecolor}{HTML}{008800}
  \definecolor{myurlcolor}{HTML}{0066cc}

  \hypersetup{
    linkcolor  = mylinkcolor,
    citecolor  = mycitecolor,
    urlcolor   = myurlcolor,
    colorlinks = true,
  }

}

\providecommand{\doi}{}
\renewcommand{\doi}[1]{\href{https://doi.org/#1}{doi:
\begingroup\urlstyle{rm}\nolinkurl{#1}\endgroup}}

\usepackage{clrscode3e}

\usepackage{paralist}

\newcommand{\graphscale}{0.75}

\usepackage{accents}
\newlength{\dhatheight}
\newcommand{\doublehat}[1]{%
  \settoheight{\dhatheight}{\ensuremath{\hat{#1}}}%
  \addtolength{\dhatheight}{-0.25ex}%
  \hat{\vphantom{\rule{1pt}{\dhatheight}}%
  \smash{\hat{#1}}}
}

\usepackage{tikz}
\usetikzlibrary{arrows,decorations.pathreplacing} 
\tikzset{node distance=1.5cm}
\tikzset{graph node/.style={circle, fill=gray!20}}
\tikzset{graph edge/.style={thick, >=stealth'}}

\numberwithin{figure}{subsection}
\numberwithin{table}{subsection}

\newrefformat{prop}{Proposition \ref{#1}}
\newrefformat{fig}{Figure \ref{#1}}
\newrefformat{tab}{Table \ref{#1}}
\newrefformat{def}{Definition \ref{#1}}
\newrefformat{exa}{Example \ref{#1}}
\newrefformat{cor}{Corollary \ref{#1}}
\newrefformat{sec}{\S\ref{#1}}
\newrefformat{sub}{\S\ref{#1}}
\newrefformat{subsec}{\S\ref{#1}}
\newrefformat{app}{Appendix \ref{#1}}

\usepackage{enumitem}

\usepackage{upref}

\usepackage{letltxmacro}
\LetLtxMacro\oldhyperref\hyperref
\renewcommand{\hyperref}[2][]{\oldhyperref[#1]{\upshape #2}}

\usepackage{dsfont}

\@ifundefined{showcaptionsetup}{}{%
\PassOptionsToPackage{caption=false}{subfig}}
\usepackage{subfig}
\makeatother

\providecommand{\corollaryname}{Corollary}
\providecommand{\definitionname}{Definition}
\providecommand{\examplename}{Example}
\providecommand{\lemmaname}{Lemma}
\providecommand{\propositionname}{Proposition}
\providecommand{\remarkname}{Remark}
\providecommand{\theoremname}{Theorem}

\begin{document}

\title{Weakly chained matrices,\\ policy iteration, and impulse control}

\date{}

\author{P. Azimzadeh\thanks{David R. Cheriton School of Computer
    Science, University of Waterloo, Waterloo ON, Canada N2L 3G1 {\tt
    \href{mailto:pazimzad@uwaterloo.ca}{pazimzad@uwaterloo.ca}} and
  {\tt \href{mailto:paforsyt@uwaterloo.ca}{paforsyt@uwaterloo.ca}}}
\and P. A. Forsyth\footnotemark[1]}
\maketitle
\begin{abstract}
  This work is motivated by numerical solutions to \emph{Hamilton-Jacobi-Bellman
  quasi-variational inequalities} (HJBQVIs) associated with \emph{combined
  stochastic and impulse control} problems. In particular, we consider
  (i) \emph{direct control}, (ii) \emph{penalized}, and (iii)
  \emph{semi-Lagrangian}
  discretization schemes applied to the HJBQVI problem. Scheme (i) takes
  the form of a \emph{Bellman problem} involving an operator which is
  not necessarily contractive. We consider the well-posedness of the
  Bellman problem and give sufficient conditions for convergence of
  the corresponding policy iteration. To do so, we use \emph{weakly
  chained diagonally dominant matrices}, which give a graph-theoretic
  characterization of nonsingular weakly diagonally dominant
  M-matrices. We compare
  schemes (i)\textendash (iii) under the following examples: (a) optimal
  control of the exchange rate, (b) optimal consumption with fixed and
  proportional transaction costs, and (c) pricing guaranteed minimum
  withdrawal benefits in variable annuities. We find that one should
  abstain from using scheme (i).
\end{abstract}
\newcommand{\mykeywords}{Hamilton-Jacobi-Bellman equation, combined
  stochastic and impulse control, policy iteration, weakly chained diagonally
  dominant matrix, optimal exchange rate, optimal consumption,
GMWB}\newcommand{\mysubjects}{65N06,
93E20}

\smallskip \noindent \textbf{AMS subject classifications.} \mysubjects

\section{Introduction}

This work is motivated by the computation of numerical solutions to
\emph{Hamilton-Jacobi-Bellman quasi-variational inequalities} (HJBQVI)
associated with combined stochastic and impulse control. These problems
are of the form:\newcommand{\myproblem}{Find a viscosity solution
  (see \cite[Definition 2.2]{ishii1993viscosity}) of the HJBQVI
  \begin{multline}
    0=F(t,x,u,Du(t,x),D^{2}u(t,x))\\
    \coloneqq-
    \begin{cases}
      \max\left(\sup_{w\in W}\left\{ {\displaystyle \frac{\partial
        u}{\partial t}}+\text{\emph{\L}}^{w}u-\rho u+f^{w}\right\}
      ,\mathcal{M}u-u\right) & \text{on
      }[0,T)\times(\overline{\Omega}\setminus\Lambda)\\
      \max\left(g-u,\mathcal{M}u-u\right) & \text{on }\partial^{+}\Omega
    \end{cases}\label{eq:hjbqvi_pde_and_boundary}
  \end{multline}
  where $\Omega\subset\mathbb{R}^{d}$ is open, $\Lambda\subset\partial\Omega$,
  $\partial^{+}\Omega\coloneqq([0,T)\times\Lambda)\cup(\{T\}\times\overline{\Omega})$,
  \emph{$\text{\L}^{w}\coloneqq\text{\L}(t,x,w)$} is the (possibly
  degenerate) generator of an SDE, $f^{w}\coloneqq f(t,x,w)$ is a forcing
  term, and $\mathcal{M}$ is the impulse (a.k.a. intervention) operator
  \begin{equation}
    \mathcal{M}u(t,x)\coloneqq\sup_{z\in Z(t,x)}\left\{
    u(t,x+\Gamma(t,x,z))+K(t,x,z)\right\}
    .\label{eq:hjbqvi_intervention_operator}
  \end{equation}
  If $Z(t,x)$ is empty at a particular point $(t,x)$, $\mathcal{M}u(t,x)$
  is understood to take the value $-\infty$, corresponding to no impulses
being allowed at that point.}
\begin{problem*} \myproblem
\end{problem*}

We focus on implicit discretization schemes for the HJBQVI problem
that do not suffer from the usual timestep restrictions of explicit
schemes. In particular, we consider (i) \emph{direct control}, (ii)
\emph{penalized} and (iii) \emph{semi-Lagrangian} schemes. The semi-Lagrangian
scheme (used for HJBQVIs in \cite{chen2008numerical}) differs from
its counterparts in that it handles controlled terms using information
from the previous timestep. As such, computing the solution of this
scheme does not require an iterative method. However, this scheme
requires that the control $w$ in $\text{\L}^{w}$ appears only in
the coefficient of the first-order term. For the other two schemes,
an iterative method is needed. The particular iterative method analyzed
herein is Howard's\emph{ policy iteration }algorithm. Not considered
is the alternative \emph{value iteration} algorithm, due to its poor
performance as the numerical grid is refined \cite[\S6.1]{forsyth2007numerical}.
Convergence of policy iteration applied to the penalized scheme turns
out to be a trivial consequence of the strict diagonal dominance of
the input matrices to policy iteration (\prettyref{subsec:penalized}).
Convergence of policy iteration applied to the direct control scheme
is a more delicate matter, as discussed below.

The direct control scheme takes the form of the fixed point problem
\begin{equation}
  \text{find }v\in\mathbb{R}^{M}\text{ such that
  }v=\max\left(\sup_{w\in\mathcal{W}}L(w)v+c(w),\sup_{z\in\mathcal{Z}}B(z)v+k(z)\right)\label{eq:fixed_point_problem}
\end{equation}
where $L(w)$ and $B(z)$ are contractive and nonexpansive matrices,
respectively. It is understood that the supremum and maximum are element-wise
and controls are ``row-decoupled'' (see $\S$\ref{sec:policy_iteration}).
\cite{chancelier2007policy} gives sufficient conditions for convergence
of a policy iteration to the unique solution of \eqref{eq:fixed_point_problem}.
However, convergence in \cite{chancelier2007policy} is conditional
on the choice of initial guess \cite[Theorem 2 (iii)]{chancelier2007policy}.
We remove this constraint.

More importantly, \cite{chancelier2007policy} restricts the admissible
set of controls and imposes a strong assumption on $B(z)$ (of which
assumption \ref{ass:path_to_sdd} in this work is an analogue) to
ensure convergence of policy iteration applied to
\eqref{eq:fixed_point_problem}.
Unfortunately, reasonable instances of problem \eqref{eq:fixed_point_problem}
(including examples in this work) do not necessarily satisfy this
condition. We show that, under a much weaker assumption, a solution
to \eqref{eq:fixed_point_problem} is unique. Moreover, when
\ref{ass:path_to_sdd}
is not satisfied directly, we provide a way to construct this solution
by considering a ``modified problem'' that satisfies \ref{ass:path_to_sdd}.
Roughly speaking, one arrives at the modified problem by removing
some suboptimal controls from the control set. However, this procedure
is ad hoc (i.e. problem dependent).

To establish the above relaxations, we use \emph{weakly chained diagonally
dominant} (WCDD) matrices. WCDD matrices give a graph-theoretic characterization
of nonsingular weakly diagonally dominant M-matrices
(\prettyref{thm:no_loop_matrix_characterization}).
The WCDD matrix approach to the convergence of policy iteration applied
to \eqref{eq:fixed_point_problem} is intuitive and established using
well-known results on policy iteration
(\prettyref{prop:convergence_of_policy_iteration}).

The ad hoc removal of suboptimal controls makes the direct control
scheme less robust than its counterparts, for which control sets need
not be altered to ensure convergence. It is thus natural to ask if
there is an advantage to using a direct control scheme. To answer
this, we apply each scheme to the following examples:
\begin{itemize}
  \item optimal control of the exchange rate;
  \item optimal consumption with fixed and proportional transaction costs;
  \item pricing guaranteed minimum withdrawal benefits in variable annuities.
\end{itemize}
The semi-Lagrangian scheme only requires a single linear solve per
timestep since no iterative method is needed. However, as mentioned
above, such a scheme cannot be used if the control $w$ appears in
the diffusion coefficient of $\text{\L}^{w}$ (or if the underlying
process is Lévy with controlled arrival rate). We find that the penalized
scheme performs at least as well as the direct control scheme. Both produce
nearly identical results and often require roughly the same amount
of computation. In the specific case of the optimal consumption problem,
the penalized scheme even outperforms the direct control scheme, taking
only a few policy iterations to converge per timestep.

We mention that in the infinite-horizon setting ($T=\infty$), optimal
consumption with fixed and proportional transaction costs was considered
numerically in \cite{chancelier2002combined} using iterated optimal
stopping, a theoretical  tool \cite[Chapter 7, Lemma 7.1]{oksendal2005applied}
for the construction of solutions that has found its way into numerical
implementations \cite{le2010finite,bayraktar2011pricing}. Computationally,
for finite-horizon problems ($T<\infty$), iterated optimal stopping
has high space complexity \cite{babbin2014comparison}, and is thus
not considered here. Also not considered here is the simulation of
penalized backward stochastic differential equations
\cite{kharroubi2010backward},
a recent alternative well-suited to high-dimensional problems.

In this work, we restrict our attention to problems of dimension three
or lower. To keep focus on the interesting aspects of impulse control,
we assume that between impulses, the underlying stochastic process
associated with the HJBQVI is a Brownian motion with drift
$\mu\coloneqq\mu(t,x,w)$
and scaling $\sigma\coloneqq\sigma(t,x,w)$ (we can extend to a Lévy
process with nontrivial arrival rate by, e.g., \cite{cont2005finite}).
This allows us to write
\[
  \text{\L}^{w}u(t,x)\coloneqq\frac{1}{2}\operatorname{trace}\left(\sigma(t,x,w)\sigma^{\intercal}(t,x,w)D_{x}^{2}u(t,x)\right)+\left\langle
  \mu(t,x,w),D_{x}u(t,x)\right\rangle .
\]

We mention here that problem \eqref{eq:fixed_point_problem} can also
be interpreted as a Bellman problem associated with an infinite-horizon
Markov decision process (MDP) with vanishing discount
(\prettyref{exa:optimal_control_of_markov_chain_with_vanishing_discount}).
In fact, \eqref{eq:fixed_point_problem} is a generalization of a
reflecting boundary problem (see, e.g., the monograph of Kushner and
Dupuis \cite[pg. 39--40]{kushner1992numerical}). In the context of
MDPs, $L(w)$ and $B(z)$ capture the transition probabilities at
states with nonvanishing and vanishing discount factors, respectively.
A WCDD matrix condition guarantees the convergence of policy iteration
to the unique solution of the Bellman problem
(\prettyref{cor:optimal_control_of_markov_chain_with_vanishing_discount}).
Intuitively, this condition ensures that the underlying Markov chain
arrives (with positive probability) at a state with nonvanishing discount
independent of the initial state.

We summarize some of our main findings below:
\begin{itemize}
  \item Policy iteration applied to a (monotone) direct control
    scheme frequently
    fails due to the possible singularity of the matrix iterates.
  \item We establish provably convergent techniques to eliminate singularity.
    However, applying these techniques is problem-dependent.
  \item We show that policy iteration applied to a (monotone) penalized scheme
    never fails. Numerical tests on three problems confirm that such a
    scheme performs at least as well as (and sometimes better than) its
    direct control counterpart.
\end{itemize}
The additional effort required to ensure the convergence in the direct
control case along with the comparable (if not better) performance
in the penalized case suggests that one should abstain from a direct
control scheme.

An outline of this work is as follows. \prettyref{sec:policy_iteration}
reminds the reader of a well-known result on the convergence of policy
iteration. \prettyref{sec:no_loop_matrices} discusses WCDD matrices.
\prettyref{sec:fixed_point_problem} gives conditions for the convergence
of policy iteration applied to problem \eqref{eq:fixed_point_problem}
and its well-posedness under weaker assumptions. A self-contained
MDP example is given therein
(\prettyref{exa:optimal_control_of_markov_chain_with_vanishing_discount}).
\prettyref{sec:numerical_schemes} introduces numerical schemes for
the HJBQVI problem \eqref{eq:hjbqvi_pde_and_boundary}, with numerical
examples given in \prettyref{sec:examples}.

\section{Policy iteration\label{sec:policy_iteration}}

In the sequel, we will see that each of the discretization schemes
for \eqref{eq:hjbqvi_pde_and_boundary} take the form of a Bellman
problem:
\begin{equation}
  \text{find }v\in\mathbb{R}^{M}\text{ such that
  }\sup_{P\in\mathcal{P}}\left\{ -A(P)v+b(P)\right\}
  =0\label{eq:bellman_problem}
\end{equation}
where $A\colon\mathcal{P}\rightarrow\mathbb{R}^{M\times M}$ and
$b\colon\mathcal{P}\rightarrow\mathbb{R}^{M}$.
It is understood that (i) $\mathcal{P}\coloneqq\prod_{i=1}^{M}\mathcal{P}_{i}$
is a finite product of nonempty sets, (ii) controls are row-decoupled:
\[
  \left[A(P)\right]_{ij}\text{ and }\left[b(P)\right]_{i}\text{
  depend only on }P_{i}\in\mathcal{P}_{i},
\]
(iii) the order on $\mathbb{R}^{M}$ (resp. $\mathbb{R}^{M\times M}$)
is element-wise:
\[
  \text{for }x,y\in\mathbb{R}^{M}\text{, }x\geq y\text{ }\text{if and
  only if }x_{i}\geq y_{i}\text{ for all }i,
\]
and (iv) the supremum is induced by this order:
\[
  \text{for }\left\{ x(P)\right\}
  _{P\in\mathcal{P}}\subset\mathbb{R}^{M}\text{,
  }\sup_{P\in\mathcal{P}}x(P)\text{ is a vector with components
  }\sup_{P\in\mathcal{P}}\left[x(P)\right]_{i}.
\]

Let $\proc{Solve}(A,b,x^{0})$ denote a call to a linear solver for
$Ax=b$ with initial guess $x^{0}$ (algebraically, $\proc{Solve}$
  computes $x$ exactly; in practice, an iterative solver is used and
the choice of $x^{0}$ affects performance). A \emph{policy iteration}
algorithm is given by:

\begin{codebox}

  \Procname{\proc{Policy-Iteration}$(\mathcal{P},A(\cdot),b(\cdot),v^{0})$}

  \li \For $\ell\gets1,2,\ldots$ \li \Do

  Pick $P^{\ell}$ such that
  $-A(P^{\ell})v^{\ell-1}+b(P^{\ell})=\sup_{P\in\mathcal{P}}\{-A(P)v^{\ell-1}+b(P)\}$

  \li    $v^{\ell}\coloneqq\proc{Solve}(A(P^{\ell}),b(P^{\ell}),v^{\ell-1})$

  \End

\end{codebox}
\begin{defn}[Monotone matrix]
  \label{def:monotone_matrix} A real square matrix $A$ is monotone
  (in the sense of Collatz) if for all real vectors $v$, $Av\geq0$
  implies $v\geq0$.
\end{defn}

We use the following assumptions:

\begin{enumerate}[label=(H\arabic{enumi}),ref=(H\arabic{enumi}),start=0]

  \item \label{ass:inverse_is_bounded} $P\mapsto A(P)^{-1}$ is bounded
    on the set $\{P\in\mathcal{P}\colon A(P)\text{ is nonsingular}\}$.

  \item \label{ass:bounded_and_argsup}
    \begin{inparaenum}[(i)]
    \item
      $A$ and $b$ are bounded and
    \item for all $x$ in $\mathbb{R}^{M}$,
      there exists $P_{x}$ in $\mathcal{P}$ such that
      $-A(P_{x})x+b(P_{x})=\sup_{P\in\mathcal{P}}\{-A(P)x+b(P)\}$.
    \end{inparaenum}

\end{enumerate}
\newcommand{\assweaker}{\hyperref[ass:bounded_and_argsup]{(H1.i)}}
\newcommand{\assargsup}{\hyperref[ass:bounded_and_argsup]{(H1.ii)}}

\begin{prop}[Convergence of policy iteration]
  \label{prop:convergence_of_policy_iteration}  Suppose
  \ref{ass:inverse_is_bounded},
  \ref{ass:bounded_and_argsup}, and that $A(P)$ is a monotone matrix
  for all $P$ in $\mathcal{P}$. $(v^{\ell})_{\ell\geq1}$ defined
  by \proc{Policy-Iteration} is nondecreasing and converges to the
  unique solution $v$ of \eqref{eq:bellman_problem}. Moreover, if
  $\mathcal{P}$ is finite, convergence occurs in at most $|\mathcal{P}|$
  iterations (i.e. $v^{|\mathcal{P}|}=v^{|\mathcal{P}|+1}=\cdots$).
\end{prop}

The monotone convergence of $(v^{\ell})_{\ell\geq1}$ to the unique
solution of \eqref{eq:bellman_problem} can be proven similarly to
\prettyref{thm:convergence_of_sequential_policy_iteration} of
\prettyref{app:sequential_policy_iteration}.
See \cite[Theorem 2.1]{bokanowski2009some} for a proof of the finite
termination when $\mathcal{P}$ is finite. In practice, $\mathcal{P}$
is often finite, in which case \ref{ass:inverse_is_bounded} and
\ref{ass:bounded_and_argsup}
are trivial.

\newcommand{\myrem}{\prettyref{thm:convergence_of_sequential_policy_iteration}
  establishes the existence and uniqueness of solutions to
  \eqref{eq:bellman_problem}
  independent of \assargsup. Owing to this, results that rely on
  \prettyref{prop:convergence_of_policy_iteration}
  can be relaxed to exclude \assargsup, with the caveat that when $\mathcal{P}$
  is infinite, $\proc{Policy-Iteration}$ be replaced by
  \proc{$\epsilon$-Policy-Iteration}
  (see \prettyref{app:sequential_policy_iteration}). In this case,
  the resulting sequence $(v^{\ell})_{\ell\geq1}$ is not necessarily
nondecreasing.}
\begin{rem}
  \myrem
\end{rem}

\section{Weakly chained diagonally dominant
matrices\label{sec:no_loop_matrices}}

We say row $i$ of a complex matrix $A\coloneqq(a_{ij})$ is strictly
diagonally dominant (SDD) if $|a_{ii}|>\sum_{j\neq i}|a_{ij}|$. We
say $A$ is SDD if all of its rows are SDD. Weakly diagonally dominant
(WDD) is defined with weak inequality instead.

\begin{defn}
  \label{def:no_loop_matrix} A complex square matrix $A$ is said to
  be a weakly chained diagonally dominant (WCDD) if:

  \begin{inparaenum}[(i)]
  \item \label{itm:no_loop_matrix_is_wdd}
    $A$ is WDD;

  \item \label{itm:no_loop_matrix_has_paths_to_sdd_rows} for each
    row $r$, there exists a path in the graph of $A$ from $r$ to an
    SDD row $p$.
  \end{inparaenum}
\end{defn}

Recall that the directed graph of an $M\times M$ complex matrix
$A\coloneqq(a_{ij})$
is given by the vertices $\{1,\ldots,M\}$ and edges defined as follows:
there exists an edge from $i$ to $j$ if and only if $a_{ij}\neq0$.
Note that if $r$ is itself an SDD row, the trivial path $r\rightarrow r$
satisfies the requirement of \eqref{itm:no_loop_matrix_has_paths_to_sdd_rows}
in the above.

The nonsingularity of WCDD matrices is proven in
\cite{shivakumar1974sufficient}.
We provide an elementary proof for the convenience of the reader:
\begin{lem}
  \label{lem:no_loop_matrix_is_nonsingular} A WCDD matrix is nonsingular.
\end{lem}

\begin{proof}
  Suppose $\lambda=0$ is an eigenvalue of $A\coloneqq(a_{ij})$. Let
  $v\neq0$ be an associated eigenvector with components of modulus
  at most unity. Let $r$ be such that $|v_{r}|=1\geq|v_{j}|$ for all
  $j$. By the Gershgorin circle theorem,
  \[
    \left|\lambda-a_{rr}\right|=\left|a_{rr}\right|\leq\sum_{j\neq
    r}\left|a_{rj}\right|\left|v_{j}\right|\leq\sum_{j\neq
    r}\left|a_{rj}\right|.
  \]
  Since $A$ is WDD, it follows that $|a_{rr}|=\sum_{j\neq r}|a_{rj}|$,
  and hence $r$ is not an SDD row. Therefore, there exists a path
  $r\rightarrow p_{1}\rightarrow\cdots\rightarrow p_{k}$
  where $p_{k}$ is an SDD row. Since
  \[
    \left|a_{rr}\right|=\sum_{j\neq
    r}\left|a_{rj}\right|\left|v_{j}\right|=\sum_{j\neq r}\left|a_{rj}\right|,
  \]
  it follows that $|v_{j}|=1$ whenever $|a_{rj}|\neq0$. Because
  $|a_{rp_{1}}|\neq0$,
  $|v_{p_{1}}|=1$. Repeating the same argument as above with $r=p_{1}$
  yields $|a_{p_{1}p_{1}}|=\sum_{j\neq p_{1}}|a_{p_{1}j}|$, and hence
  $p_{1}$ is not an SDD row. Continuing the procedure, $p_{k}$ is
  not SDD, a contradiction.
\end{proof}

We recall some well-known classes of matrices:
\begin{defn}
  \label{def:z_matrix} A Z-matrix is a real matrix with nonpositive
  off-diagonals.
\end{defn}

\begin{defn}
  \label{def:m_matrix} A nonsingular M-matrix is a monotone Z-matrix.
\end{defn}

We are now ready to state a fundamental characterization of
nonsingular WDD M-matrices:
\begin{thm}[Characterization theorem]
  \label{thm:no_loop_matrix_characterization} The following are equivalent:

  \begin{inparaenum}[(i)]
  \item $A$ is a WCDD Z-matrix with positive
    diagonals; \label{itm:no_loop_z_matrix_positive_diagonals}

  \item $A$ is a nonsingular WDD M-matrix. \label{itm:wdd_m_matrix}

  \end{inparaenum}
\end{thm}
\begin{proof}
  Since a nonsingular WDD Z-matrix with positive diagonals is an M-matrix
  (a consequence of, e.g., \cite[Theorem 1.$A_3$]{plemmons1977m}),
  \eqref{itm:no_loop_z_matrix_positive_diagonals} implies
  \eqref{itm:wdd_m_matrix}
  follows by \prettyref{lem:no_loop_matrix_is_nonsingular}.

  As for the converse, since a nonsingular M-matrix has positive
  diagonal elements
  (a consequence of, e.g., \cite[Theorem 1.$K_{35}$]{plemmons1977m}),
  it is sufficient to show that a WDD Z-matrix $A\in\mathbb{R}^{n\times n}$
  with positive diagonals not satisfying \prettyref{def:no_loop_matrix}
  \eqref{itm:no_loop_matrix_has_paths_to_sdd_rows} is singular. Let
  $R\subset\{1,\ldots,n\}$ be the set of rows $r$ of $A$ violating
  \prettyref{def:no_loop_matrix}
  \eqref{itm:no_loop_matrix_has_paths_to_sdd_rows}.
  Due to our assumptions, there is at least one such row, and hence
  $R$ is nonempty. Without loss of generality, we can assume $R=\{1,\ldots,m\}$
  for some $1\leq m\leq n$ (otherwise, reorder $A$). Let $e\in\mathbb{R}^{m}$
  denote the column vector whose elements are all unity. If $m=n$,
  each row sum of $A$ is zero (i.e., $Ae=0$), implying that $A$ is
  singular. If $m<n$, $A$ has the block structure\[
    A = \left(
      \begin{array}{c|c}
        B & 0 \\
        \hline
        C & D
    \end{array} \right)
    \text{ where } B \in \mathbb{R}^{m\times m}.
  \] Because rows that violate \prettyref{def:no_loop_matrix}
  \eqref{itm:no_loop_matrix_has_paths_to_sdd_rows}
  were ``isolated'' to the block $B$, the partition above ensures
  that $D$ is WCDD. Therefore, by \prettyref{lem:no_loop_matrix_is_nonsingular},
  the linear system $Dx=-Ce$ has a unique solution $x$. Moreover,
  since the row sums of $B$ are zero, $Be=0$. It follows that
  \[
    A
    \begin{pmatrix}e\\
      x
    \end{pmatrix}=
    \begin{pmatrix}Be\\
      Ce+Dx
    \end{pmatrix}=0,
  \]
  and hence $A$ is singular.
\end{proof}

This characterization is tight: a nonsingular M-matrix need not be WCDD (e.g.
  {\scriptsize$
    \begin{pmatrix}1 & -2\\
      0 & \phantom{-}1
\end{pmatrix}$}).

We mention that \eqref{itm:no_loop_z_matrix_positive_diagonals} implies
\eqref{itm:wdd_m_matrix} of \prettyref{thm:no_loop_matrix_characterization}
appears in \cite{bramble1964finite}. Therein, WCDD Z-matrices with
positive diagonals are referred to as \emph{matrices of positive type.
}To the authors' best knowledge, the converse does not appear in the
literature.

\section{The fixed point problem
\eqref{eq:fixed_point_problem}\label{sec:fixed_point_problem}}

\subsection{Convergence of policy iteration}

We assume $\mathcal{W}\coloneqq\prod_{i=1}^{M}\mathcal{W}_{i}$ and
$\mathcal{Z}\coloneqq\prod_{i=1}^{M}\mathcal{Z}_{i}$ appearing in
problem \eqref{eq:fixed_point_problem} are finite products of nonempty
sets. Let
\begin{equation}
  \mathcal{P}\coloneqq{\textstyle
  \prod_{i=1}^{M}}\mathcal{P}_{i}\text{ where
  }\mathcal{P}_{i}\coloneqq\mathcal{W}_{i}\times\mathcal{Z}_{i}\times\mathcal{D}_{i}\text{
  and }\emptyset\neq\mathcal{D}_{i}\subset\left\{ 0,1\right\}
  .\label{eq:fixed_point_problem_control_set}
\end{equation}
We associate with each $\psi\coloneqq(\psi_{1},\ldots,\psi_{M})$
in $\mathcal{D}\coloneqq\prod_{i=1}^{M}\mathcal{D}_{i}$ a diagonal
matrix $\Psi\coloneqq\operatorname{diag}(\psi)$. We use $\psi$ and
$\Psi$ interchangeably. We write $P\coloneqq(w,z,\psi)\in\mathcal{P}$
where $w\in\mathcal{W}$ and $z\in\mathcal{Z}$. We can transform
problem \eqref{eq:fixed_point_problem} into the form \eqref{eq:bellman_problem}
by taking
\begin{align}
  A(P) &
  \coloneqq\left(I-\Psi\right)\left(I-L(w)\right)+\delta\Psi\left(I-B(z)\right)\nonumber
  \\
  b(P) & \coloneqq\left(I-\Psi\right)c(w)+\delta\Psi
  k(z)\label{eq:fixed_point_problem_system}
\end{align}
where $\delta=1$ ($L(w)$ and $B(z)$ are matrices; $c(w)$ and $k(z)$
are vectors). To keep the material general, we henceforth assume the
less restrictive condition $\delta>0$ instead of $\delta=1$. Before
considering the well-posedness of
\begin{equation}
  \text{problem }\eqref{eq:bellman_problem}\text{ subject to
  }\eqref{eq:fixed_point_problem_control_set}\text{ and
  }\eqref{eq:fixed_point_problem_system},\label{eq:problem_combined}
\end{equation}
we establish that the set of solutions to \eqref{eq:problem_combined}
is independent of the choice of $\delta$:
\begin{lem}
  \label{lem:invariance}$v$ is a solution of \eqref{eq:problem_combined}
  with $\delta=1$ if and only if it is a solution of \eqref{eq:problem_combined}
  with arbitrary $\delta=\delta_{0}>0$.
\end{lem}

A proof of the above is given in \prettyref{app:proof_of_invariance}.
In the sequel, we exploit the fact that policy iteration may converge
more rapidly for particular choices of $\delta$. We now visit, as
a motivating example, an infinite-horizon MDP with vanishing discount:

\newcommand{\myexample}{\label{exa:optimal_control_of_markov_chain_with_vanishing_discount}
  Let $(X^{n})_{n\geq0}$ be a controlled homogeneous Markov chain on
  a finite state space $\{1,\ldots,M\}$. A control at state $i$ is
  a member of $\mathcal{P}_{i}$ in \eqref{eq:fixed_point_problem_control_set}
  and written $P_{i}\coloneqq(w_{i},z_{i},\psi_{i})$. The transition
  probabilities of the Markov chain are
  \[
    \mathbb{P}(X^{n+1}=j\mid X^{n}=i)=
    \begin{cases}
      w_{ij} & \text{if }\psi_{i}=0\\
      z_{ij} & \text{if }\psi_{i}=1
    \end{cases}
  \]
  where $w_{i}\coloneqq(w_{i1},\ldots,w_{iM})\geq0$ and $w_{i}e=1$
  (similarly for $z_{i}$). That is, members of $\mathcal{W}_{i}$ and
  $\mathcal{Z}_{i}$ are $M$-dimensional probability vectors. Let
  \begin{equation}
    v_{i}\coloneqq\sup_{P\in\mathcal{P}}\mathbb{E}\left[\sum_{n=0}^{\infty}U(X^{n},P)\prod_{m=0}^{n-1}D(X^{m},P)\,\biggr\vert\,X^{0}=i\right]\text{
    for all }1\leq i\leq
    M\label{eq:optimal_control_of_markov_chain_with_vanishing_discount}
  \end{equation}
  where
  \[
    U(i,P)\coloneqq
    \begin{cases}
      c_{i}(w_{i}) & \text{if }\psi_{i}=0\\
      k_{i}(z_{i}) & \text{if }\psi_{i}=1
    \end{cases}\text{ and }D(i,P)\coloneqq
    \begin{cases}
      1/\left(1+\rho\right) & \text{if }\psi_{i}=0;\\
      1 & \text{if }\psi_{i}=1.
    \end{cases}
  \]
  In the above, $\rho>0$ is a discount factor. \cite[Lemma
  5]{chancelier2007policy}
  establishes that the dynamic programming equation associated with
  \eqref{eq:optimal_control_of_markov_chain_with_vanishing_discount}
  is exactly \eqref{eq:problem_combined} with $[L(w)]_{ij}\coloneqq
  w_{ij}/(1+\rho)$,
  $[B(z)]_{ij}\coloneqq z_{ij}$, $[c(w)]_{i}\coloneqq c_{i}(w_{i})$,
and $[k(z)]_{i}\coloneqq k_{i}(z_{i})$.}
\begin{example}
  \myexample
\end{example}

In the above, states $i$ on which $\psi_{i}=1$ are the ``trouble''
states with vanishing discount factor. In fact, requiring $\psi_{i}=0$
for all $i$ returns us to a nonvanishing discount factor problem
whose well-posedness is easy to establish.

The following assumptions will prove paramount:

\begin{enumerate}[label=(H\arabic{enumi}),ref=(H\arabic{enumi}),start=2]

  \item \label{ass:path_to_sdd} For each $P\coloneqq(w,z,\psi)$ in
    $\mathcal{P}$ and state $i$ with $\psi_{i}=1$, there exists a path
    in the graph of $B(z)$ from $i$ to some state $j(i)$ with $\psi_{j(i)}=0$.

  \item \label{ass:sdd_and_wdd} For each $P\coloneqq(w,z,\psi)$ in
    $\mathcal{P}$, $I-L(w)$ is an SDD Z-matrix with positive diagonals
    and $I-B(z)$ is a WDD Z-matrix whose diagonals satisfy
    $0\leq[B(z)]_{ii}\leq1$.

\end{enumerate}
\begin{thm}
  \label{thm:fixed_point_problem} Suppose
  \ref{ass:inverse_is_bounded}\emph{\textendash }\ref{ass:sdd_and_wdd}.
  $(v^{\ell})_{\ell\geq1}$ defined by $\proc{Policy-Iteration}$ is
  nondecreasing and converges to the unique solution $v$ of
  \eqref{eq:problem_combined}.
  Moreover, if $\mathcal{P}$ is finite, convergence occurs in at most
  $|\mathcal{P}|$ iterations.
\end{thm}
\begin{proof}
  \ref{ass:path_to_sdd} and \ref{ass:sdd_and_wdd} ensure that $A(P)$
  is a WCDD Z-matrix with positive diagonals. The desired result follows
  from \prettyref{thm:no_loop_matrix_characterization} and
  \prettyref{prop:convergence_of_policy_iteration}.
\end{proof}

\begin{cor}
  \label{cor:optimal_control_of_markov_chain_with_vanishing_discount}
  Consider Example
  \ref{exa:optimal_control_of_markov_chain_with_vanishing_discount}.
  Suppose \ref{ass:inverse_is_bounded}\emph{\textendash }\ref{ass:path_to_sdd}.
  $(v^{\ell})_{\ell\geq1}$ defined by $\proc{Policy-Iteration}$ converges
  to $v$ in $\mathbb{R}^{M}$ satisfying
  \eqref{eq:optimal_control_of_markov_chain_with_vanishing_discount}.
\end{cor}

An example satisfying \ref{ass:path_to_sdd} is given:

\begin{example} \label{exa:unidirectional} Consider Example
  \ref{exa:optimal_control_of_markov_chain_with_vanishing_discount}.
  Suppose all $P\coloneqq(w,z,\psi)$ in $\mathcal{P}$ satisfy
  \begin{equation}
    \psi_{1}=0\text{ and }z_{ij}=0\text{ if }1<i\leq j.\label{eq:unidirectional}
  \end{equation}
  This corresponds to (i) a nonvanishing discount at state $1$ and
  that (ii) transitions from a state with vanishing discount are unidirectional
  (if $\psi_{X^{n}}=1$, $X^{n+1}<X^{n}$ a.s.). See
  \prettyref{fig:unidirectional}
  for example graphs of $B(z)$ subject to \eqref{eq:unidirectional}.

\end{example}

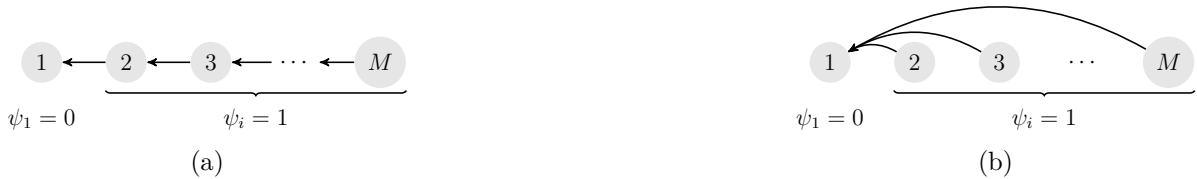
\begin{figure}
  \begin{centering}
    \subfloat[\label{fig:unidirectional_multiple_hops}]{
      \begin{centering}
        \scalebox{\graphscale}{
          \begin{tikzpicture}

            \node [graph node] (1) {1};
            \node [graph node, right of=1] (2) {2};
            \node [graph node, right of=2] (3) {3};
            \node [right of=3] (dots) {$\cdots$};
            \node [graph node, right of=dots] (M) {$M$};

            \draw[graph edge, ->] (2) to (1);
            \draw[graph edge, ->] (3) to (2);
            \draw[graph edge, ->] (dots) to (3);
            \draw[graph edge, ->] (M) to (dots);

            \draw [thick, decoration={brace, raise=0.5cm}, decorate]
            (M.east) -- (2.west) node [pos=0.5, yshift=-1cm] {$\psi_i=1$};
            \node [below of=1, node distance=1cm] {$\psi_1=0$};

          \end{tikzpicture}
        }
        \par
      \end{centering}
    }\hfill{}\subfloat[\label{fig:undirectional_one_hop}]{
      \begin{centering}
        \scalebox{\graphscale}{
          \begin{tikzpicture}

            \node [graph node] (1) {1};
            \node [graph node, right of=1] (2) {2};
            \node [graph node, right of=2] (3) {3};
            \node [right of=3] (dots) {$\cdots$};
            \node [graph node, right of=dots] (M) {$M$};

            \draw[graph edge, ->] (2) to [bend right] (1);
            \draw[graph edge, ->] (3) to [bend right] (1);
            \draw[graph edge, ->] (M) to [bend right] (1);

            \draw [thick, decoration={brace, raise=0.5cm}, decorate]
            (M.east) -- (2.west) node [pos=0.5, yshift=-1cm] {$\psi_i=1$};
            \node [below of=1, node distance=1cm] {$\psi_1=0$};

          \end{tikzpicture}
        }
        \par
      \end{centering}
    }
    \par
  \end{centering}
  \caption{Graphs of possible matrices $B(z)$ from Example
  \ref{exa:unidirectional}\label{fig:unidirectional}}
\end{figure}

\subsection{Uniqueness}

Let
\begin{equation}
  \mathbb{L}v\coloneqq\sup_{w\in\mathcal{W}}\left\{
  L(w)v+c(w)\right\} \text{ and
  }\mathbb{B}v\coloneqq\sup_{z\in\mathcal{Z}}\left\{
  B(z)v+k(z)\right\} .\label{eq:fixed_point_problem_operators}
\end{equation}
If $1 \notin \mathcal{D}_i$ for some state $i$, the $i$-th row of $B(\cdot)$
and $i$-th component of $k(\cdot)$ do not affect the problem.
Therefore, without loss of generality, we assume henceforth
\[
  [B(\cdot)]_{ij} = [I]_{ij} \text{ and } [k(\cdot)]_i = -1
  \text{ whenever } 1 \notin \mathcal{D}_i.
\]
The condition \ref{ass:path_to_sdd} turns out to be too restrictive
for some problems of interest. However, the following weaker property
of $\mathbb{B}$ is not unusual:

\begin{enumerate}[label=(H\arabic{enumi}),ref=(H\arabic{enumi}),start=4]

  \item \label{ass:subidempotence} For each solution $v$ of
    \eqref{eq:problem_combined}
    and each state $i$, there exist integers $m(i)$ and $n(i)$ such
    that $0\leq n(i)<m(i)$ and
    $[\mathbb{B}^{m(i)}v]_{i}<[\mathbb{B}^{n(i)}v]_{i}$.

\end{enumerate}
\begin{lem}
  \label{lem:bounded_longest_path}Suppose \ref{ass:sdd_and_wdd} and
  \ref{ass:subidempotence}. Let
  $(P^{\ell})_{\ell\geq0}\coloneqq(w^{\ell},z^{\ell},\psi^{\ell})_{\ell\geq0}$
  be a sequence in $\mathcal{P}$ and $v$ a solution of
  \eqref{eq:problem_combined}
  satisfying
  \[
    -A(P^{\ell})v+b(P^{\ell})\rightarrow\sup_{P\in\mathcal{P}}\left\{
    -A(P)v+b(P)\right\} =0.
  \]
  There exists $\ell_{0}\geq0$ such that for each $\ell\geq\ell_{0}$
  and state $i$ with $\psi_{i}^{\ell}=1$, there exists a path in the
  graph of $B(z^{\ell})$ from $i$ to some state $j(i,\ell)$ with
  $\psi_{j(i,\ell)}^{\ell}=0$.
\end{lem}
\begin{proof}
  Suppose the contrary. A pigeonhole principle argument yields the existence
  of a subsequence
  $(P^{\ell_{q}})_{q\geq0}\coloneqq(w^{\ell_{q}},z^{\ell_{q}},\psi^{\ell_{q}})_{q\geq0}$
  of $(P^{\ell})_{\ell\geq0}$ such that
  \begin{itemize}
    \item $\psi^{\ell_{q}}=\psi$
      is a constant independent of $q$;
    \item the graph of $B(z^{\ell_{q}})$
      (call it $G$) is a constant independent of $q$;
    \item there exists
      $i$ such that $\psi_{i}=1$ and for all $j$ reachable from $i$
      (in $G$), $\psi_{j}=1$.
  \end{itemize} Let $V\coloneqq\{j_{1},\ldots,j_{k}\}$
  be the states reachable from $i$, including $i$ itself. Let $r\in
  V$ be arbitrary.
  Since the limit of a convergent sequence equals to the limit of any
  of its subsequences,
  \[
    \left[B(z^{\ell_{q}})v+k(z^{\ell_{q}})\right]_{r}-v_{r}=\frac{1}{\delta}\left[-A(P^{\ell_{q}})v+b(P^{\ell_{q}})\right]_{r}\rightarrow0,
  \]
  and hence $[\mathbb{B}v]_{r}\geq v_{r}$. Now, since $r\in V$
  implies $\psi_{r}=1$ and $v$ is a solution
  of \eqref{eq:problem_combined}, it follows that $[\mathbb{B}v]_{r}\leq v_{r}$.
  Therefore, $[\mathbb{B}v]_{r}=v_{r}$. Moreover,
  \begin{multline*}
    [\mathbb{B}^{2}v]_{r}=[\mathbb{B}(\mathbb{B}v)]_{r}\geq[B(z^{\ell_{q}})(\mathbb{B}v)+k(z^{\ell_{q}})]_{r}=\sum_{j\in
    V}[B(z^{\ell_{q}})]_{rj}[\mathbb{B}v]_{j}+[k(z^{\ell_{q}})]_{r}\\
    =\sum_{j\in
    V}[B(z^{\ell_{q}})]_{rj}v_{j}+[k(z^{\ell_{q}})]_{r}=[B(z^{\ell_{q}})v+k(z^{\ell_{q}})]_{r}\rightarrow[\mathbb{B}v]_{r}
  \end{multline*}
  and hence $[\mathbb{B}^{2}v]_{r}\geq[\mathbb{B}v]_{r}$.
  For each state $j$, if $1 \in \mathcal{D}_j$, the fact that $v$ is
  a solution implies $[\mathbb{B}v]_j \leq v_j$.
  If instead $1 \notin \mathcal{D}_j$, our convention yields
  $[\mathbb{B}v]_j = v_j - 1 < v_j$.
  Therefore, $\mathbb{B}v \leq v$.
  Since $\mathbb{B}$
  is a monotone operator by \ref{ass:sdd_and_wdd},
  $\mathbb{B}^{2}v\leq\mathbb{B}v$, and hence
  $[\mathbb{B}^{2}v]_{r}=[\mathbb{B}v]_{r}$.
  We can continue this procedure to obtain
  \[
    v_{r}=[\mathbb{B}v]_{r}=[\mathbb{B}^{2}v]_{r}=[\mathbb{B}^{3}v]_{r}=\cdots
  \]
  Setting $r=i$ in the above yields a contradiction to \ref{ass:subidempotence}.
\end{proof}

If we take the trivial path $i\rightarrow i$ as having length zero,
the proof above also implies that for $\ell$ sufficiently large and
for all $i$, there is a path of length $<m(i)$ (where $m(i)$ is
specified by \ref{ass:subidempotence}) in the graph of $B(z^{\ell})$
from $i$ to some state $j(i,\ell)$ with $\psi_{j(i,\ell)}^{\ell}=0$.
An example is given below:

\begin{example} \label{exa:subidempotence} Consider Example
  \ref{exa:optimal_control_of_markov_chain_with_vanishing_discount}
  with $\mathcal{Z}_{i}=\mathcal{Z}_{j}$ for all states $i$ and $j$.
  For all states $i$, suppose $1 \in \mathcal{D}_i$ and let
  $k_{i}(z_{i})\coloneqq-C<0$. It follows
  that for all $x$ in $\mathbb{R}^{M}$,
  \[
    \mathbb{B}^{2}x=\sup_{z,z^{\prime}\in\mathcal{Z}}\left\{
    B(z)B(z^{\prime})x\right\} -2C<\sup_{z\in\mathcal{Z}}\left\{
    B(z)x\right\} -C=\mathbb{B}x,
  \]
  so that \ref{ass:subidempotence} is satisfied with $1=n(i)<m(i)=2$
  for all $i$. Intuitively, the controller pays twice the cost to apply
  $\mathbb{B}$ twice.

  In this case, denoting by $v$ a solution of \eqref{eq:problem_combined},
  the control shown in \prettyref{fig:unidirectional_multiple_hops}
  cannot correspond to some $P_{v}$ satisfying
  $-A(P_{v})v+b(P_{v})=\sup_{P\in\mathcal{P}}\{-A(P)v+b(P)\}=0$
  since any path from $i>2$ to $j=1$ is of length at least $m(i)=2$.

\end{example}

We can now prove uniqueness independent of \ref{ass:path_to_sdd}:
\begin{thm}
  \label{thm:uniqueness} Suppose \ref{ass:inverse_is_bounded},
  \ref{ass:sdd_and_wdd},
  and \ref{ass:subidempotence}. A solution of \eqref{eq:problem_combined}
  is unique.
\end{thm}
\begin{proof}
  Let $x$ and $y$ be two solutions and $(P^{\ell})_{\ell\geq0}$ be
  a sequence in $\mathcal{P}$ such that
  \[
    -A(P^{\ell})y+b(P^{\ell})\rightarrow\sup_{P\in\mathcal{P}}\left\{
    -A(P)y+b(P)\right\} =0.
  \]
  It follows from \ref{ass:sdd_and_wdd}, \ref{ass:subidempotence},
  and \prettyref{lem:bounded_longest_path} that we can, without loss
  of generality, assume $A(P^{\ell})$ is a WCDD Z-matrix with positive
  diagonals, and hence an M-matrix by
  \prettyref{thm:no_loop_matrix_characterization}.
  For some sequence $(\epsilon^{\ell})_{\ell\geq0}$ in $\mathbb{R}^{M}$
  with $\epsilon^{\ell}\rightarrow0$, we can write
  \[
    -A(P^{\ell})y+b(P^{\ell})+\epsilon^{\ell}=0=\sup_{P\in\mathcal{P}}\left\{
    -A(P)x+b(P)\right\} \geq-A(P^{\ell})x+b(P^{\ell}),
  \]
  so that $A(P^{\ell})(x-y)\geq-\epsilon^{\ell}$. Since the inverse
  of a monotone matrix has nonnegative elements and $P\mapsto A(P)^{-1}$
  is bounded by \ref{ass:inverse_is_bounded}, $x-y\geq0$. Similarly,
  $y-x\geq0$.
\end{proof}

Unfortunately, the conditions of \prettyref{thm:uniqueness} cannot
guarantee that the iterates $(v^{\ell})_{\ell\geq1}$ given by policy
iteration are well-defined, as $A(P^{\ell})$ may be singular for
some $\ell\geq1$. This is demonstrated in the following example,
for which \ref{ass:path_to_sdd} does not hold:

\begin{example}[Failure of policy iteration] \label{exa:policy_iteration_fails}
  Consider Example
  \ref{exa:optimal_control_of_markov_chain_with_vanishing_discount}.
  For all states $i$, suppose $1 \in \mathcal{D}_i$, let
  $\mathcal{Z}_{i}\coloneqq\{e^{j}\}_{j=1}^{M}$
  be the set of standard basis vectors, and let
  \[
    k_{i}(z_{i})\coloneqq-C-\sum_{j}z_{ij}\left|i-j\right|<0\text{ where }C>0.
  \]
  As in Example \ref{exa:subidempotence}, \ref{ass:subidempotence}
  is satisfied due to the fixed cost $C$.

  Let $\delta\coloneqq1$. Suppose there exists a state $r$ with
  $c_{r}(w_{r})<-C$ for all controls $w$ in $\mathcal{W}$. It
  is readily verified that $\proc{Policy-Iteration}$ initialized with
  the zero vector $v^{0}\coloneqq0$ picks $P^{1}\coloneqq(w^{1},z^{1},\psi^{1})$
  with $z_{r}^{1}=e^{r}$ and $\psi_{r}^{1}=1$. It follows that
  \[
    [A(P^{1})]_{rj}=[I-B(z^{1})]_{rj}=[I]_{rj}-[I]_{rj}=0\text{ for all }j
  \]
  so that $A(P^{1})$ is singular, and hence $v^{1}$ is undefined.

\end{example}

For any $\ell\geq1$, it is possible to construct more complicated
examples in which the matrices $A(P^{1}),\ldots,A(P^{\ell-1})$ are
nonsingular while $A(P^{\ell})$ is singular. That is, policy iteration
can fail at any iterate.

\subsection{Policy iteration on a modified problem}

As demonstrated in the previous section, if \ref{ass:path_to_sdd}
is not satisfied, policy iteration may fail. We may, however, hope
to construct a solution by performing policy iteration on a ``modified
problem'' with control set $\mathcal{P}^{\prime}$ obtained by removing
controls $P$ in $\mathcal{P}$ that render $A(P)$ singular.

We define \ref{ass:bounded_and_argsup}$^{\prime}$ by replacing all
occurrences of $\mathcal{P}$ with $\mathcal{P}^{\prime}$ in the
definition of \ref{ass:bounded_and_argsup}. \ref{ass:path_to_sdd}$^{\prime}$
and \ref{ass:sdd_and_wdd}$^{\prime}$ are defined similarly. We can
now state the above idea precisely:
\begin{thm}
  \label{thm:the_standard_machinery} Let
  $\mathcal{P}^{\prime}\coloneqq{\textstyle
  \prod_{i=1}^{M}}\mathcal{P}_{i}^{\prime}$
  where each $\mathcal{P}_{i}^{\prime}\subset\mathcal{P}_{i}$ is nonempty.
  Suppose \ref{ass:inverse_is_bounded}, \ref{ass:bounded_and_argsup}$^{\prime}$,
  \ref{ass:path_to_sdd}$^{\prime}$, \ref{ass:sdd_and_wdd},
  \ref{ass:subidempotence},
  and
  \begin{equation}
    \text{for all }v\text{ in }\mathbb{R}^{M},\text{
    }\sup_{P\in\mathcal{P}^{\prime}}\left\{ -A(P)v+b(P)\right\}
    =0\implies\sup_{P\in\mathcal{P}}\left\{ -A(P)v+b(P)\right\}
    =0.\label{eq:modified_solves_original}
  \end{equation}
  $(v^{\ell})_{\ell\geq1}$ defined by
  $\proc{Policy-Iteration}(\mathcal{P}^{\prime},A(\cdot),b(\cdot),v^{0})$
  is nondecreasing and converges to the unique solution of
  \eqref{eq:problem_combined}.
  Moreover, if $\mathcal{P}^{\prime}$ is finite, convergence occurs
  in at most $|\mathcal{P}^{\prime}|$ iterations.
\end{thm}
\begin{proof}
  Since $\mathcal{P}^{\prime}\subset\mathcal{P}$, it follows immediately
  that \ref{ass:inverse_is_bounded}$^{\prime}$ and
  \ref{ass:sdd_and_wdd}$^{\prime}$
  are satisfied, so that by \prettyref{thm:fixed_point_problem},
  $(v^{\ell})_{\ell\geq1}$
  is well-defined and converges to the unique solution $v$ of the modified
  problem. That is, $v^{\ell}\rightarrow v$ and
  $\sup_{P\in\mathcal{P}^{\prime}}\left\{ -A(P)v+b(P)\right\} =0$.
  By \eqref{eq:modified_solves_original}, $v$ solves the original
  problem \eqref{eq:problem_combined}. Since solutions to
  \eqref{eq:problem_combined}
  are unique by \prettyref{thm:uniqueness}, the desired result follows.
\end{proof}

We now give a nontrivial example (in the sense that \ref{ass:path_to_sdd}
fails) for which we can apply \prettyref{thm:the_standard_machinery}:

\begin{example} \label{exa:the_standard_machinery} Consider Example
  \ref{exa:optimal_control_of_markov_chain_with_vanishing_discount}.
  Let (i) $\mathcal{Z}$ and $k$ be given as in Example
  \ref{exa:policy_iteration_fails},
  (ii) \label{itm:three_point_stencil} $\mathcal{W}_{i}$ be the set
  of all $M$-dimensional probability vectors $w_{i}$ with $w_{ij}=0$
  whenever $|i-j|>1$, (iii) $c$ be continuous and bounded, and (iv)
  $\overline{c}\coloneqq\max_{w\in\mathcal{W}}c(w)$ with
  $\overline{c}_{i-1}\geq\overline{c}_{i}$
  for all $1<i\leq M$. Suppose $\mathcal{D}_i = \{0, 1\}$ for all
  $1\leq i\leq M$ and let $\mathcal{P}^{\prime}$ be all $P\coloneqq(w,z,\psi)$
  in $\mathcal{P}$ satisfying \eqref{eq:unidirectional}. Then, the
  conditions of \prettyref{thm:the_standard_machinery} are satisfied.

\end{example}
\begin{proof}
  It is straightforward to verify \ref{ass:inverse_is_bounded},
  \ref{ass:bounded_and_argsup}$^{\prime}$,
  \ref{ass:path_to_sdd}$^{\prime}$, \ref{ass:sdd_and_wdd}, and
  \ref{ass:subidempotence}.
  Thus, it is sufficient to show \eqref{eq:modified_solves_original}.
  We write
  $\mathcal{P}_{i}^{\prime}\coloneqq\mathcal{W}_{i}\times\mathcal{Z}_{i}^{\prime}\times\mathcal{D}_{i}^{\prime}$
  and define $[\mathbb{B}^{\prime}x]_{i}$ for $i>1$ by replacing $\mathcal{Z}$
  with $\mathcal{Z}^{\prime}\coloneqq\prod_{i=1}^{M}\mathcal{Z}_{i}^{\prime}$
  in \eqref{eq:fixed_point_problem_operators}.

  We first show that the solution $v$ to the modified problem is nonincreasing:
  \[
    v_{i-1}\geq v_{i}\text{ for all }1<i\leq M.
  \]
  Suppose the contrary. Let $r>1$ be the minimal element such that
  $v_{r-1}<v_{r}$. If $v_{r}=[\mathbb{B}^{\prime}v]_{r}$, then
  $v_{r}=v_{j}-C-|r-j|$
  for some $j<r$. Either $j=r-1$ or
  \[
    v_{r-1}\geq\left[\mathbb{B}^{\prime}v\right]_{r-1}\geq
    v_{j}-C-\left|\left(r-1\right)-j\right|\geq v_{r}
  \]
  (both are contradictions). It follows that $v_{r}=[\mathbb{L}v]_{r}$.
  Letting $w_{0j},v_{0}\coloneqq0$ for notational convenience, assumption
  (ii) implies
  \[
    v_{r-1}\geq\left[\mathbb{L}v\right]_{r-1}=\max_{w\in\mathcal{W}}\left\{
    \sum_{j=r-2}^{r}\frac{w_{r-1,j}}{1+\rho}v_{j}+\left[c(w)\right]_{r-1}\right\}
    \geq\frac{v_{r-1}}{1+\rho}+\overline{c}_{r-1}
  \]
  so that $v_{r-1}\geq(1+1/\rho)\overline{c}_{r-1}$. If $v_{r}\geq v_{r+1}$,
  it follows similarly from $v_{r}=[\mathbb{L}v]_{r}$ that
  $v_{r}\leq(1+1/\rho)\overline{c}_{r}$
  so that $v_{r-1}\geq v_{r}$ (a contradiction) and hence it must be
  the case that $v_{r}<v_{r+1}$. We can repeat this argument inductively
  to arrive at the contradiction
  \[
    v_{r-1}<v_{r}<\cdots<v_{M}\text{ and
    }v_{r-1}\geq(1+1/\rho)\overline{c}_{r-1}\geq(1+1/\rho)\overline{c}_{M}\geq
    v_{M}.
  \]

  Since $v$ is nonincreasing, $v\geq\mathbb{B}v$ (it is suboptimal
    to take $\psi_{i}=1$ and $z_{ij}=1$ for states $i$ and $j$ with
  $j\geq i$), and hence \eqref{eq:modified_solves_original} holds.
\end{proof}

\section{Numerical schemes for the HJBQVI problem\label{sec:numerical_schemes}}

All numerical schemes herein are on a rectilinear grid
\[
  \left\{ t^{1},\ldots,t^{N}\right\} \times\left\{
  x_{1}^{1},x_{2}^{1},\ldots\right\} \times\cdots\times\left\{
  x_{1}^{d},x_{2}^{d},\ldots\right\}
\]
where $0\eqqcolon t^{1}<\cdots<t^{N}\coloneqq T$ and
$x_{1}^{j}<x_{2}^{j}<\cdots$
for all $j$. Multi-indices are used (i.e.
$x_{i}\coloneqq(x_{i_{1}},\ldots,x_{i_{d}})$).
$M$ denotes the number of spatial points $x_{i}$. For functions
$q\coloneqq q(t,x)$ defined on $[0,T]\times\mathbb{R}^{d}$, the
shorthands $q_{i}^{n}\coloneqq q(t^{n},x_{i})$ and $q^{n}(x)\coloneqq
q(t^{n},x)$
are employed. In the absence of ambiguity, we use $q^{n}$ to denote
the vector with components $q_{i}^{n}$ and take $\Delta t\coloneqq
t^{n+1}-t^{n}$.
It is understood that $\max_{n}\{t^{n+1}-t^{n}\}\rightarrow0$ and
$\max_{i}\{x_{i+1}-x_{i}\}\rightarrow0$ as $h\rightarrow0$, where
$h$ denotes a ``global'' discretization parameter that controls
the coarseness of the grid.

Control sets $W$ and $Z(t,x)$ are approximated by finite sets
$\emptyset\neq W_{h}\subset W$
and $Z_{h}(t,x)\subset Z(t,x)$. The reader concerned with consistency
should impose some regularity to justify this approximation, such
as: (i) $W$ is compact, (ii) $Z$ is everywhere compact and continuous
with respect to the Hausdorff metric, and (iii) $\max_{w\in
W}\min_{w_{h}\in W_{h}}\left|w-w_{h}\right|\rightarrow0$
as the discretization parameter $h\rightarrow0$ along with an identical
pointwise condition for $Z$ and $Z_{h}$.

The discretized impulse operator \eqref{eq:hjbqvi_intervention_operator}
is
\[
  \left[\mathcal{M}_{h}^{n}u^{n}\right]_{i}\coloneqq\max_{z\in(Z_{h})_{i}^{n}}\left\{
    u^{n}\left\llbracket x_{i}+\Gamma_{i}^{n}(z)\right\rrbracket
  +K_{i}^{n}(z)\right\}
\]
where $\varphi\llbracket x\rrbracket$ denotes linear interpolation
using the value of $\varphi$ on grid nodes. It is understood that
controls $z$ that cause $x_{i}+\Gamma_{i}^{n}(z)$ to exit the numerical
grid are not included in $(Z_{h})_{i}^{n}$. We use $\text{\L}_{h}^{n}(w)$
to denote a consistent discretization of $\text{\L}^{w}$ with coefficients
frozen at $t=t^{n}$.

Recall that in \eqref{eq:hjbqvi_pde_and_boundary},
$\Lambda\subset\partial\Omega$
is a special subset of the boundary at which a Dirichlet-like condition
is applied. To distinguish points, we denote by $\Phi$ a diagonal
matrix satisfying $[\Phi]_{ii}=0$ whenever $x_{i}$ is in $\Lambda$
and $[\Phi]_{ii}=1$ otherwise.

Since the Dirichlet-like condition is imposed at the final time $t=T$,
the numerical method proceeds backwards in time (i.e. from $t^{n+1}$
to $t^{n}$). Assuming interventions do not improve the payoff
(i.e. $g^{N}\geq\mathcal{M}_{h}^{N}g^{N}$), let $u_{i}^{N}\coloneqq g_{i}^{N}$.
The numerical solution $u^{n}$ at timestep $1\leq n<N$ produced
by each scheme (given the solution at the previous timestep, $u^{n+1}$)
is written as a solution of \eqref{eq:bellman_problem} with $A$
and $b$ picked appropriately. Control sets are given by
\eqref{eq:fixed_point_problem_control_set}
and
\begin{equation}
  \mathcal{W}_{i}\coloneqq W_{h}\text{, }\mathcal{Z}_{i}\coloneqq
  \begin{cases}
    (Z_{h})_{i}^{n_{0}} & \text{if }(Z_{h})_{i}^{n_{0}}\neq\emptyset\\
    \left\{ \emptyset\right\}  & \text{otherwise}
  \end{cases}\text{, and }\mathcal{D}_{i}\coloneqq
  \begin{cases}
    \left\{ 0,1\right\}  & \text{if }(Z_{h})_{i}^{n_{0}}\neq\emptyset\\
    \left\{ 0\right\}  & \text{otherwise}
  \end{cases}\label{eq:control_set_for_numerical_schemes}
\end{equation}
where $n_{0}$ is $n+1$ for the semi-Lagrangian scheme (see
$\S$\ref{subsec:explicit_control})
and $n$ otherwise. As a technical detail, we take $\mathcal{Z}_{i}$
to be a nonempty set (we choose $\{\emptyset\}$ arbitrarily) whenever
$(Z_{h})_{i}^{n_{0}}$ is empty to ensure that the product
$\mathcal{W}_{i}\times\mathcal{Z}_{i}\times\mathcal{D}_{i}$
of \eqref{eq:fixed_point_problem_control_set} is nonempty.

We make the following assumptions:

\begin{enumerate}[label=(A\arabic{enumi}),ref=(A\arabic{enumi}),start=0]

  \item
    \label{ass:finite_control_set}$\mathcal{W}\coloneqq\prod_{i=1}^{M}\mathcal{W}_{i}$
    and $\mathcal{Z}\coloneqq\prod_{i=1}^{M}\mathcal{Z}_{i}$ are finite.

  \item \label{ass:discrete_degenerate_elliptic_operator_is_wdd} For
    all $w$ in $\mathcal{W}$, $-\text{\L}_{h}^{n}(w)$ is a WDD Z-matrix
    with nonnegative diagonals.

  \item \label{ass:discrete_intervention_operator_is_markov_matrix}
    For all $z$ in $\mathcal{Z}$, $B^{n}(z)$ is a right stochastic
    (a.k.a. Markov) matrix with $[B^{n}(z)v]_{i}=v\llbracket
    x_{i}+\Gamma_{i}^{n}(z_{i})\rrbracket$.

  \item \label{ass:scheme_constants} $\rho\geq0$ and $\delta,\epsilon>0$.

\end{enumerate} Since \ref{ass:finite_control_set} ensures that
$\mathcal{P}$ is finite, all schemes in the sequel satisfy
\ref{ass:inverse_is_bounded}
and \ref{ass:bounded_and_argsup}.

\begin{rem}Barles and Souganidis \cite{barles1991convergence} prove
  that a numerical scheme converges to the unique viscosity solution
  of a fully nonlinear second order equation (such as
  \eqref{eq:hjbqvi_pde_and_boundary})
  satisfying a comparison result\emph{ }if it is monotone in the viscosity
  sense, \emph{$\ell_{\infty}$ }stable, and consistent. Comparison
  results for the HJBQVI \eqref{eq:hjbqvi_pde_and_boundary} are provided
  in \cite[Theorem 5.11]{seydel2010general}.
  \ref{ass:discrete_degenerate_elliptic_operator_is_wdd}
  and \ref{ass:discrete_intervention_operator_is_markov_matrix} ensure
  monotonicity (see \cite[Section 1.3]{oberman2006convergent} for an
  example of a stable nonmonotone scheme that fails to converge). For
  brevity, we do not give proofs of consistency or discuss stability
  here.
\end{rem}

\subsection{Direct control\label{subsec:direct_control}}

In a direct control formulation, either the generator ($\sup_{w\in
W}\{\partial u/\partial t+\text{\L}^{w}u-\rho u+f^{w}\}$)
or impulse ($\mathcal{M}u-u$) component is active at any grid point.
Since these have different units, comparing them in floating point
arithmetic requires a scaling factor $\delta>0$ to ensure fast convergence
\cite{huang2013inexact} (see also \prettyref{lem:invariance}). Scaling
by $\delta$ and discretizing \eqref{eq:hjbqvi_pde_and_boundary}
(ignoring boundary conditions) yields
\[
  \max\biggl(\max_{w\in W_{h}}\left\{
      \frac{u_{i}^{n+1}-u_{i}^{n}}{\Delta
      t}+\left[\text{\L}_{h}^{n}(w)u^{n}\right]_{i}-\rho
    u_{i}^{n}+f_{i}^{n}(w)\right\}
  ,\delta\left(\left[\mathcal{M}_{h}^{n}u^{n}\right]_{i}-u_{i}^{n}\right)\biggr)=0.
\]
Including boundary conditions, this is put in the form of
\eqref{eq:problem_combined}
by taking
\begin{align}
  L(w) & \coloneqq\Phi\left(\text{\L}_{h}^{n}(w)-\rho I\right)\Delta
  t; & c(w) & \coloneqq\Phi\left(u^{n+1}+f^{n}(w)\Delta
  t\right)+\left(I-\Phi\right)g^{n};\nonumber \\
  B(z) & \coloneqq B^{n}(z); & k(z) & \coloneqq
  K^{n}(z).\label{eq:direct_control_system}
\end{align}
With $B$ and $k$ given above, the operator $\mathcal{M}_{h}^{n}$
is equivalent to $\mathbb{B}$ defined in
\eqref{eq:fixed_point_problem_operators}.

$L$ and $B$ given above satisfy \ref{ass:sdd_and_wdd} due to
\ref{ass:discrete_degenerate_elliptic_operator_is_wdd}\textendash
\ref{ass:scheme_constants}.
Therefore, \ref{ass:subidempotence} is a sufficient condition for
uniqueness of solutions (\prettyref{thm:uniqueness}). Similarly,
\ref{ass:path_to_sdd} is a sufficient condition for convergence of
the corresponding policy iteration (\prettyref{thm:fixed_point_problem}).

\subsection{Penalized\label{subsec:penalized}}

A penalized formulation (treated in detail in \cite{witte2012penalty})
imposes a penalty scaled by $1/\epsilon^{\prime}\gg0$ whenever $\mathcal{M}u>u$.
The scheme is given by:
\[
  \max_{w\in W_{h}}\left\{ \frac{u_{i}^{n+1}-u_{i}^{n}}{\Delta
    t}+\left[\text{\L}_{h}^{n}(w)u^{n}\right]_{i}-\rho
  u_{i}^{n}+f_{i}^{n}(w)\right\}
  +\max\left(\left[\mathcal{M}_{h}^{n}u^{n}\right]_{i}-u_{i}^{n},0\right)/\epsilon^{\prime}=0.
\]
For simplicity, we take $\epsilon^{\prime}\coloneqq\epsilon\Delta t$
for some $\epsilon>0$. Including boundary conditions, this is put
in the form \eqref{eq:bellman_problem} by taking
\begin{align*}
  A(P) & \coloneqq I+\Phi\left(\rho
  I-\text{\L}_{h}^{n}(w)\right)\Delta t+\Psi\left(I-B^{n}(z)\right)/\epsilon;\\
  b(P) & \coloneqq\Phi\left(u^{n+1}+f^{n}(w)\Delta
  t\right)+\left(I-\Phi\right)g^{n}+\Psi K^{n}(z)/\epsilon.
\end{align*}
Convergence of the corresponding policy iteration is trivial since
$A(P)$ is an SDD Z-matrix with positive diagonals (by virtue of
  \ref{ass:discrete_degenerate_elliptic_operator_is_wdd}\textendash
\ref{ass:scheme_constants}),
and hence an M-matrix.

\subsection{Semi-Lagrangian\label{subsec:explicit_control}}

The crux of a semi-Lagrangian scheme is the use of a \emph{Lagrangian
derivative} to remove the $D_{x}$ coefficient's dependency on the
control $w$. It is assumed that (i) $\sigma$ is independent of the
control and (ii) the drift $\mu$ and forcing term $f$ can be split
into (sufficiently regular) controlled and uncontrolled components:
\[
  \mu(y,w)=\hat{\mu}(y)+\doublehat{\mu}(y,w)\text{ and
  }f(y,w)=\hat{f}(y)+\doublehat{f}(y,w).
\]
We now give some intuition behind a semi-Lagrangian scheme. Consider
a generator $\hat{\text{\L}}$ corresponding to an uncontrolled SDE:
\[
  \hat{\text{\L}}u(y)\coloneqq\text{\L}(w)u(y)-\left\langle
  \doublehat{\mu}(y,w),D_{x}u(y)\right\rangle .
\]
Letting $X\coloneqq X(t)$ denote a $d$-dimensional trajectory satisfying
\[
  X(t^{n})=x_{i}\text{ and }dX(t)=\doublehat{\mu}(t,X(t),w)dt\text{
  on }(t^{n},t^{n+1}]
\]
so that $X(t^{n+1})\approx
X(t^{n})+\doublehat{\mu}(t^{n},X(t^{n}),w)\Delta
t=x_{i}+\doublehat{\mu}_{i}^{n}(w)\Delta t$,
we define the Lagrangian derivative\emph{ }with respect to $X$ as
\[
  \frac{Du}{Dt}(t,X(t),w)\coloneqq\frac{\partial}{\partial
  t}\left[u(t,X(t))\right]=\frac{\partial u}{\partial
  t}(t,X(t))+\left\langle
  \doublehat{\mu}(t,X(t),w),D_{x}u(t,X(t))\right\rangle .
\]
Ignoring boundary conditions, we substitute ${\displaystyle \frac{Du}{Dt}}$
into \eqref{eq:hjbqvi_pde_and_boundary} to get
\[
  \max\left(\sup_{w\in W}\left\{
    \frac{Du}{Dt}^{w}+\hat{\text{\L}}u-\rho u+f^{w}\right\}
  ,\mathcal{M}u-u\right)=0.
\]
A discretization of the above is
\begin{multline*}
  \max\left(\max_{w\in W_{h}}\left\{ u^{n+1}\left\llbracket
      x_{i}+\doublehat{\mu}_{i}^{n}(w)\Delta t\right\rrbracket
    +\doublehat{f}_{i}^{n+1}(w)\Delta t\right\}
  ,\left[\mathcal{M}_{h}^{n+1}u^{n+1}\right]_{i}\right)\\
  -u_{i}^{n}+\left(\left[\hat{\text{\L}}_{h}^{n}u^{n}\right]_{i}-\rho
  u_{i}^{n}+\hat{f}_{i}^{n}\right)\Delta t=0.
\end{multline*}
It is understood that controls $w$ that cause
$x_{i}+\doublehat{\mu}_{i}^{n}(w)\Delta t$
to exit the numerical grid are not considered at node $i$. Consistency
of this scheme (subject to some mild assumptions) can be shown similarly
to \cite[Lemma 6.6]{chen2008numerical}.

In lieu of \ref{ass:discrete_degenerate_elliptic_operator_is_wdd},
we assume:

\begin{enumerate}[label=(A\arabic{enumi}$^\prime$),ref=(A\arabic{enumi}$^\prime$),start=1]

  \item \label{ass:explicit_control_scheme_wdd} $-\hat{\text{\L}}_{h}^{n}$
    is a WDD Z-matrix with nonnegative diagonals.

\end{enumerate} Let $\vec{x}$ denote a vector with components $x_{i}$.
Including boundary conditions, this is put in the form
\eqref{eq:bellman_problem}
by taking
\begin{align*}
  A & \coloneqq I+\Phi\left(\rho I-\hat{\text{\L}}_{h}^{n}\right)\Delta t;\\
  b(P) & \coloneqq\Phi\left(\hat{f}^{n}\Delta
    t+\left(I-\Psi\right)\left(u^{n+1}\left\llbracket
      \vec{x}+\doublehat{\mu}^{n}(w)\Delta t\right\rrbracket
  +\doublehat{f}^{n+1}(w)\Delta t\right)\right)\\
  &
  \hspace{14em}+\left(I-\Phi\right)\left(I-\Psi\right)g^{n}+\Psi\left(B^{n+1}(z)u^{n+1}+K^{n+1}(z)\right).
\end{align*}
Since $A$ is independent of $P$, \eqref{eq:bellman_problem} becomes
$Av=\max_{P\in\mathcal{P}}\{b(P)\}$; no iterative method is required.
$A$ is nonsingular since it is SDD (by virtue of
  \ref{ass:explicit_control_scheme_wdd}
and \ref{ass:scheme_constants}).

\section{Examples\label{sec:examples}}

The remainder of this work focuses on numerical examples.

\subsection{Optimal combined control of the exchange
rate\label{subsec:exchange_rate}}

The following is studied in \cite{mundaca1998optimal,cadenillas1999optimal}.
Consider a government able to influence the foreign exchange (FEX)
rate of its currency by:
\begin{itemize}
  \item choosing the domestic interest rate \textbf{(stochastic control)};
  \item buying or selling foreign currency \textbf{(impulse control)}.
\end{itemize}
Let $(r_{t})_{t\geq0}$ denote the domestic interest rate process
and $\overline{r}$ the foreign interest rate. At any point in time,
the government can buy ($z>0$) or sell ($z<0$) foreign currency
to influence the FEX market. $(X_{t})_{t\geq0}$, the log of the FEX
rate, follows
\begin{align*}
  dX_{t} & =-a\left(r_{t}-\overline{r}\right)dt+\sigma
  d\mathfrak{W}_{t} & \text{ if }\tau_{j}<t<\tau_{j+1} &  &
  \text{(stochastic control)};\\
  X_{\tau_{j+1}} & =X_{\tau_{j+1}-}+z_{\tau_{j+1}} &  &  &
  \text{(impulse control)}.
\end{align*}
$(\mathfrak{W}_{t})_{t\geq0}$ is a standard Brownian motion. $a>0$
parameterizes the effect of the interest rate differential,
$w_{t}\coloneqq r_{t}-\overline{r}$,
on the FEX rate.

Let
$\theta\coloneqq(w,\tau_{1},\tau_{2},\ldots,z_{\tau_{1}},z_{\tau_{2}},\ldots)$
where (i) $(w_{t})_{t\geq0}$ is an adapted process, (ii)
$\tau_{1},\tau_{2},\ldots$
are stopping times with
$0\eqqcolon\tau_{0}\leq\tau_{1}\leq\tau_{2}\leq\ldots\leq \infty$,
and (iii) $z_{\tau_{k}}$ is a $\tau_{k}$-measurable random variable
taking values from some set $Z(\tau_{k},X_{\tau_{k}})$. Any such
$\theta$ satisfying these properties is referred to as a \emph{combined
control}.

A combined control is \emph{admissible }if at all times $t$,
$w_{\text{min}}\leq w_{t}\leq w_{\text{max}}$
(alternatively, we could impose this up to null sets). Let $\Theta$
denote the set of all admissible controls. The optimal cost at time
$t$ when $X_{t}=x$ is given by
\begin{equation}
  u(t,x)\coloneqq e^{\rho
  t}\sup_{\theta\in\Theta}\mathbb{E}^{(t,x)}\left[-\int_{t}^{T}e^{-\rho
    s}\left(p(X_{s})+bw_{s}^{2}\right)ds-\sum_{\tau_{j}\leq
  T}e^{-\rho\tau_{j}}\left(\lambda\left|z_{\tau_{j}}\right|+C\right)\right].\label{eq:exchange_rate_value}
\end{equation}
The cost of the distance of the FEX rate to the optimal parity $x^{\star}$
is parameterized by the function $p$. We take $p(x)\coloneqq(x-x^{\star})^{2}$.
The constant $b\geq0$ parameterizes the cost associated with a nonzero
interest rate differential. $\lambda\geq0$ and $C>0$ parameterize
the cost of an impulse. $\rho\geq0$ is a discount factor.

It is well-known \cite{bensoussan1984impulse} that the dynamic programming
equation associated to \eqref{eq:exchange_rate_value} is the HJBQVI
on $\Omega\coloneqq\mathbb{R}$ and $\Lambda\coloneqq\emptyset$ given
by \eqref{eq:hjbqvi_pde_and_boundary} with $g(T,x)\coloneqq0$ and
\begin{table}
  \begin{centering}
    {\footnotesize%
      \begin{tabular}{rrrl}
        \toprule
        \multicolumn{2}{r}{Parameter} & \multicolumn{2}{l}{Value}\tabularnewline
        \midrule
        Discount factor & $\rho$ & 2\% & per annum\tabularnewline
        Volatility & $\sigma$ & 30\% & per annum\tabularnewline
        Expiry & $T$ & 10 & years\tabularnewline
        Optimal parity & $x^{\star}$ & 0 & \tabularnewline
        Interest rate differential & $W$ & $[-7\%,7\%]$ & per
        annum\tabularnewline
        Interest rate differential effect & $a$ & 0.25 & \tabularnewline
        Interest rate differential cost & $b$ & 3 & \tabularnewline
        Scaled transaction cost & $\lambda$ & 1 & \tabularnewline
        Fixed transaction cost & $C$ & 0.1 & \tabularnewline
        \bottomrule
    \end{tabular}}
    \par
  \end{centering}
  \caption{Optimal combined control of the exchange rate:
  parameters\label{tab:exchange_rate_parameters}}
\end{table}
\begin{align*}
  W & \coloneqq[w_{\text{min}},w_{\text{max}}]; & Z(t,x) &
  \coloneqq\mathbb{R};\\
  \text{\L}(t,x,w) & \coloneqq-aw\frac{\partial u}{\partial
  x}+\frac{1}{2}\sigma^{2}\frac{\partial^{2}u}{\partial x^{2}}; &
  \Gamma(t,x,z) & \coloneqq z;\\
  f(t,x,w) & \coloneqq-p(x)-bw^{2}; & K(t,x,z) & \coloneqq-\lambda|z|-C.
\end{align*}

\subsubsection{Convergence of the direct control
scheme\label{subsec:exchange_rate_direct_control}}

We truncate $[0,T]\times\mathbb{R}$
to $[0,T]\times[x_{1},x_{M}]$ using a spatial grid symmetric about
$x^{\star}$, where $x_{1}=x^{\star}-R$ and $x_{M}=x^{\star}+R$
for some $R>0$. We also truncate $Z(t,x)=\mathbb{R}$
to $[x_{1},x_{M}]-x$ so that the exchange rate after an impulse,
$x+\Gamma(t,x,z)=x+z$, remains in the computational domain. Let
$\Delta z>0$ divide $x_{M}-x_{1}$.
A discretization of the truncated $Z(t,x)$ is
\[
  \left(Z_{h}^{n}\right)_{i}\coloneqq\left\{ 0,\Delta z,2\Delta
  z,\ldots,x_{M}-x_{1}\right\} +\left(x_{1}-x_{i}\right).
\]

Artificial boundary conditions $\partial^q u/\partial x^q=0$
are used at $x_{1}$ and $x_{M}$ so that the first and last rows of
$\text{\L}_{h}^{n}(w)$ are zero. In particular, we assume an upwind
three-point stencil \cite[Appendix C]{forsyth2007numerical} so that
\[
  \left[\text{\L}_{h}^{n}(w)v\right]_{i}\coloneqq
  \begin{cases}
    0 & \text{if }i=1\text{ or }i=M\\
    \left(v_{i-1}-v_{i}\right)\alpha_{i}^{n}(w)+\left(v_{i+1}-v_{i}\right)\beta_{i}^{n}(w)
    & \text{otherwise}
  \end{cases}
\]
where $\alpha_{i}^{n}(w)$ and $\beta_{i}^{n}(w)$ are nonnegative
constants arising from the discretization.

The direct control problem is given by \eqref{eq:problem_combined}
subject to \eqref{eq:control_set_for_numerical_schemes} and
\eqref{eq:direct_control_system}.
It is easy to verify that $\mathbb{B}^{2}x<\mathbb{B}x$ for all $x$
so that \ref{ass:subidempotence} is satisfied (recall
$\mathbb{B}=\mathcal{M}_{h}^{n}$).
By \prettyref{thm:uniqueness}, solutions to the problem are unique.
However, policy iteration may fail since \ref{ass:path_to_sdd} is
not satisfied. A trivial example violating \ref{ass:path_to_sdd}
is that of a cycle between two nodes $x_{i}\neq x_{j}$ (e.g.
  $x_{i}+\Gamma(t,x_{i},x_{j}-x_{i})=x_{j}$
and $x_{j}+\Gamma(t,x_{j},x_{i}-x_{j})=x_{i}$).

Assume that the target parity $x^{\star}$ is a grid point.
We perform policy iteration on a modified problem with control set
$\mathcal{P}^{\prime}\subsetneq\mathcal{P}$ consisting of all controls
$P\coloneqq(w,z,\psi)$ in $\mathcal{P}$ satisfying
\[
  \psi_{i}=0\text{ whenever }x_{i}=x^{\star}\text{ and
  }0<z_{i}/(x^{\star}-x_{i})\leq1\text{ whenever }\psi_{i}=1.
\]
That is, away from the target, impulses move the state toward but
not beyond the target. Under this restriction, every row with
$\psi_{i}=1$ has a directed
path to the strictly diagonally dominant target row and hence
\ref{ass:path_to_sdd}$^{\prime}$ holds. Assume also that
the discrete stochastic control set is symmetric (i.e., $W_{h}=-W_{h}$).
If $u^{n+1}$ is symmetric
about $x^{\star}$ and nonincreasing as $\left|x_{i}-x^{\star}\right|$
increases, we can use
the arguments analogous to those in Example \ref{exa:the_standard_machinery}
to establish that the solution $v=u^{n}$ of the modified problem
solves the original problem (i.e. \eqref{eq:modified_solves_original}
is satisfied) and has the same symmetry and monotonicity properties.
Since $u^{N}=0$ has these properties, induction establishes that
each $u^{n}$ does as well.

\begin{rem}The above restriction appeals to intuition: an intervention
  should move the FEX rate toward the target parity rather than away
  from it.
\end{rem}

\subsubsection{Optimal control}

\begin{figure}
  \begin{centering}
    \subfloat[Optimal control at initial time\label{fig:exchange_rate_control}]{
      \begin{centering}
        \includegraphics[height=2in]{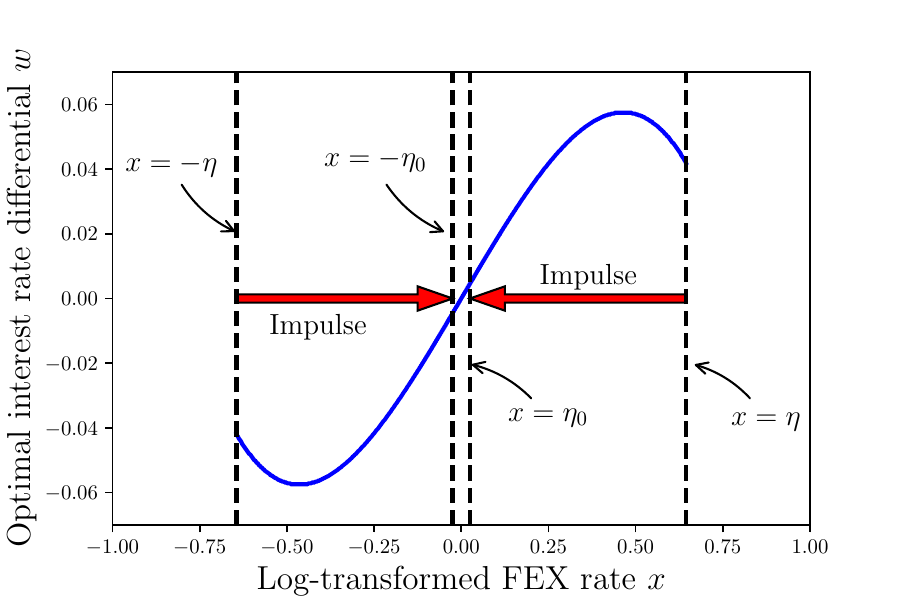}
        \par
      \end{centering}
    }\subfloat[Cost at initial time\label{fig:exchange_rate_value}]{
      \begin{centering}
        \includegraphics[height=2in]{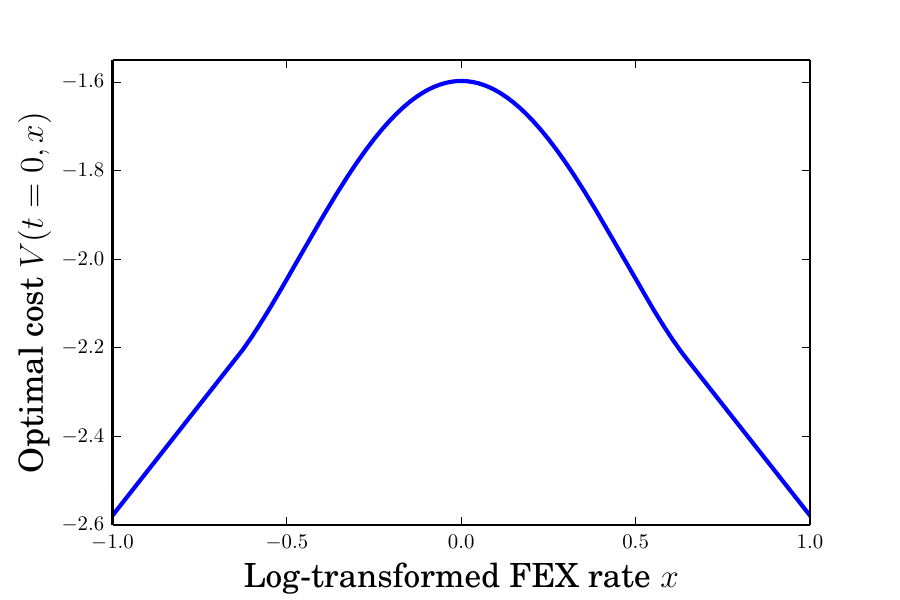}
        \par
      \end{centering}
    }
    \par
  \end{centering}
  \caption{Optimal combined control of the exchange rate at initial time}
\end{figure}
If the FEX rate is sufficiently close to the target parity $x^{\star}=0$,
the government does not intervene in the FEX market. That is, at time $t$,
impulses do not occur in $(-\eta(t),\eta(t))$ for some $\eta(t)>0$
(the region
  on which the impulse is not applied is referred to as the \emph{continuation
  region}, corresponding to nodes $i$ with $\psi_{i}=0$ in the numerical
solution). When the FEX rate at time $t$ enters $[\eta(t),\infty)$,
the government intervenes to bring it back to $\eta_{0}(t)<\eta(t)$.
When it enters $(-\infty,-\eta(t)]$, the government intervenes
symmetrically to bring it back to $-\eta_{0}(t)$.
This phenomenon is shown in \prettyref{fig:exchange_rate_control}.
The optimal cost $u$ is shown in \prettyref{fig:exchange_rate_value}.

\subsubsection{Convergence tests}

Convergence tests are shown in \prettyref{tab:exchange_rate_convergence}.
Times are normalized to the fastest semi-Lagrangian solve. The ratio
of successive changes in the solution (at a point) is reported.

BiCGSTAB with an ILUT preconditioner is used for the $\proc{Solve}$
routine (line 3 of $\proc{Policy-Iteration}$) in this and all subsequent
sections. In the specific case of the semi-Lagrangian scheme for the
exchange rate problem, a simple tridiagonal solve can be used since
the problem is a one-dimensional diffusion.
\begin{table}
  \begin{centering}
    {\footnotesize%
      \begin{tabular}{ccccc}
        \toprule
        $h$ & $x$ nodes & $w$ nodes & $z$ nodes & Timesteps\tabularnewline
        \midrule
        1 & 32 & 8 & 16 & 16\tabularnewline
        1/2 & 64 & 16 & 32 & 32\tabularnewline
        $\vdots$ & $\vdots$ & $\vdots$ & $\vdots$ & $\vdots$\tabularnewline
        \bottomrule
    \end{tabular}}
    \par
  \end{centering}
  \caption{Optimal combined control of the exchange rate: numerical grid}
\end{table}
\begin{table}
  \begin{centering}
    \subfloat[Direct control]{
      \begin{centering}
        {\footnotesize%
          \begin{tabular}{cccccc}
            \toprule
            $h$ & $u(t=0,x=0)$ & Avg. policy its. & Avg. BiCGSTAB
            its. & Ratio & Norm. time\tabularnewline
            \midrule
            1 & -1.59470667276 & 2.50 & 1.65 &  & 1.24e+01\tabularnewline
            1/2 & -1.60161214854 & 2.53 & 1.82 &  & 4.30e+01\tabularnewline
            1/4 & -1.60009885637 & 2.33 & 2.80 & -4.56 & 2.53e+02\tabularnewline
            1/8 & -1.59882094629 & 2.33 & 2.96 & 1.18 & 1.88e+03\tabularnewline
            1/16 & -1.59796763572 & 2.36 & 2.97 & 1.50 & 1.73e+04\tabularnewline
            1/32 & -1.59753341756 & 2.34 & 2.95 & 1.97 & 1.42e+05\tabularnewline
            1/64 & -1.59730416122 & 2.34 & 2.90 & 1.89 & 1.04e+06\tabularnewline
            \bottomrule
        \end{tabular}}
        \par
      \end{centering}
    }
    \par
  \end{centering}
  \begin{centering}
    \subfloat[Penalized]{
      \begin{centering}
        {\footnotesize%
          \begin{tabular}{cccccc}
            \toprule
            $h$ & $u(t=0,x=0)$ & Avg. policy its. & Avg. BiCGSTAB
            its. & Ratio & Norm. time\tabularnewline
            \midrule
            1 & -1.59542996288 & 2.56 & 1.95 &  & 1.23e+01\tabularnewline
            1/2 & -1.60176266672 & 2.53 & 2.13 &  & 4.70e+01\tabularnewline
            1/4 & -1.60012316809 & 2.34 & 2.08 & -3.86 & 2.87e+02\tabularnewline
            1/8 & -1.59883787204 & 2.33 & 3.08 & 1.28 & 2.09e+03\tabularnewline
            1/16 & -1.59796948734 & 2.36 & 3.19 & 1.48 & 1.83e+04\tabularnewline
            1/32 & -1.59753376608 & 2.35 & 3.17 & 1.99 & 1.42e+05\tabularnewline
            1/64 & -1.59730437362 & 2.34 & 3.18 & 1.90 & 1.04e+06\tabularnewline
            \bottomrule
        \end{tabular}}
        \par
      \end{centering}
    }
    \par
  \end{centering}
  \begin{centering}
    \subfloat[Semi-Lagrangian ]{
      \begin{centering}
        {\footnotesize%
          \begin{tabular}{cccc}
            \toprule
            $h$ & $u(t=0,x=0)$ & Ratio & Norm. time\tabularnewline
            \midrule
            1 & -1.21009825238 &  & 1.00e+00\tabularnewline
            1/2 & -1.40343492151 &  & 7.14e+00\tabularnewline
            1/4 & -1.50140778899 & 1.97 & 5.69e+01\tabularnewline
            1/8 & -1.54909952448 & 2.05 & 4.44e+02\tabularnewline
            1/16 & -1.57273173354 & 2.02 & 3.33e+03\tabularnewline
            1/32 & -1.58474899304 & 1.97 & 2.84e+04\tabularnewline
            1/64 & -1.59084952538 & 1.97 & 2.34e+05\tabularnewline
            \bottomrule
        \end{tabular}}
        \par
      \end{centering}
    }
    \par
  \end{centering}
  \caption{Optimal combined control of the exchange rate: convergence
  tests\label{tab:exchange_rate_convergence}}
\end{table}

$\proc{Policy-Iteration}$ is terminated upon achieving a desired
error tolerance:
\[
  \max_{i}\left\{
  \frac{\left|v_{i}^{k}-v_{i}^{k-1}\right|}{\max\left(\left|v_{i}^{k}\right|,\text{scale}\right)}\right\}
  <\text{tol}.
\]
The scale parameter ensures that unrealistic levels of accuracy are
not imposed on the solution. We take $\text{tol}=10^{-6}$ and $\text{scale}=1$
for this and all future tests. The initial guess $v^{0}$ is taken
to be the solution at the previous timestep, $u^{n+1}$.

Following \cite{huang2013inexact}, we take $\epsilon\coloneqq D\Delta t$
and $\delta\coloneqq1/\epsilon$ with $D=10^{-2}$.

For completeness, we mention that the obvious splitting with
$\hat{\mu}(t,x)\coloneqq0$
and $\hat{f}(t,x)\coloneqq-p(x)$ is used in the semi-Lagrangian scheme.
The numerical examples of the sequel (\prettyref{subsec:optimal_consumption}
and \prettyref{subsec:gmwb}) also use the obvious splittings.

Unsurprisingly, the direct control and penalized schemes are near-identical
in performance and accuracy since the scaling and penalty factors
are chosen identically (i.e. $\delta=1/\epsilon$). We mention that
the  choice of $\delta=1$ (i.e. no scaling) yields poor performance
in the direct control setting (see \cite{huang2013inexact} for an
explanation).

\subsection{Optimal consumption and portfolio with both fixed and proportional
transaction costs\label{subsec:optimal_consumption}}

The following is studied in \cite{chancelier2002combined}. Consider
an investor that, at any point in time, has two investment opportunities:
a stock and a bank account. Let $(S_{t})_{t\geq0}$ and $(B_{t})_{t\geq0}$
denote the amount of money invested in these two, respectively. The
investor is able to
\begin{itemize}

  \item consume continuously \textbf{(stochastic control)};

  \item transfer money from the bank to the stock (or vice versa) subject
    to a transaction cost \textbf{(impulse control)}.

\end{itemize}

Denote by $(w_{t})_{t\geq0}$ the consumption rate with $0\leq
w_{t}\leq w_{\text{max}}$.
At any point in time, the investor can move money to ($z>0$) or from
($z<0$) the stock incurring a transaction cost of $\lambda|z|+C$
where $C>0$ and $0\leq\lambda<1$. This is captured by
\begin{align*}
  dS_{t} & =\mu S_{t}dt+\xi S_{t}d\mathfrak{W}_{t} & \text{ if
  }\tau_{j}<t<\tau_{j+1} &  & \text{(stochastic control)};\\
  dB_{t} & =\left(rB_{t}-w_{t}\right)dt & \text{ if
  }\tau_{j}<t<\tau_{j+1} &  & \text{(stochastic control)};\\
  S_{\tau_{j+1}} & =S_{\tau_{j+1}-}+z_{\tau_{j+1}} &  &  &
  \text{(impulse control)};\\
  B_{\tau_{j+1}} &
  =B_{\tau_{j+1}-}-z_{\tau_{j+1}}-\lambda\left|z_{\tau_{j+1}}\right|-C
  &  &  & \text{(impulse control)}.
\end{align*}
A combined control
$\theta\coloneqq(w,\tau_{1},\tau_{2},\ldots,z_{\tau_{1}},z_{\tau_{2}},\ldots)$
is admissible if at all times, the stock holdings and bank account
are nonnegative. Let $\Theta$ denote the set of all admissible controls.

The investor's maximal expected utility at time $t$ with amount $S_{t}=s$
in the stock and $B_{t}=b$ in the bank account is given by
\[
  u(t,s,b)\coloneqq e^{\rho
  t}\sup_{\theta\in\Theta}\mathbb{E}^{(t,s,b)}\left[\int_{t}^{T}e^{-\rho
    t^{\prime}}\frac{w_{t^{\prime}}^{\gamma}}{\gamma}dt^{\prime}+e^{-\rho
  T}\frac{\max\left(B_{T}+\left(1-\lambda\right)S_{T}-C,0\right)^{\gamma}}{\gamma}\right]
\]
where $0\leq1-\gamma<1$ is the investor's relative risk-aversion
and $\rho\geq0$ is the rate of time preference. The utility received
at the expiry corresponds to liquidating the asset and consuming everything
instantaneously.

The associated HJBQVI on $\Omega\coloneqq(0,\infty)^{2}$ and
$\Lambda\coloneqq\emptyset$
is given by \eqref{eq:hjbqvi_pde_and_boundary} with
$g(T,x)\coloneqq\max(b+(1-\lambda)s-C,0)^{\gamma}/\gamma$
and
\begin{alignat*}{2}
  W & \coloneqq\left[0,w_{\text{max}}\right]; & Z(t,x) &
  \coloneqq\left\{ z\colon x+\Gamma(t,x,z)\geq0\right\} ;\\
  \text{\L}^{w} &
  \coloneqq\frac{1}{2}\xi^{2}s^{2}\frac{\partial^{2}}{\partial
  s^{2}}+\mu s\frac{\partial}{\partial s}+
  \begin{cases}
    \left(rb-w\right)\frac{\partial}{\partial b} & \text{if }b>0;\\
    0 & \text{otherwise;}
  \end{cases} & \Gamma(t,x,z) & \coloneqq(z,-z-\lambda|z|-C);\\
  f^{w} & \coloneqq
  \begin{cases}
    w^{\gamma}/\gamma & \text{if }b>0;\\
    0 & \text{otherwise;}
  \end{cases} & K(t,x,z) & \coloneqq0.
\end{alignat*}
In the above, expressions such as $s\cdot\partial/\partial s$ are
to be interpreted as identically zero when $s=0$. The convention
$[q_{1},q_{2}]=\emptyset$ if $q_{1}>q_{2}$ is used.
\begin{table}
  \begin{centering}
    {\footnotesize%
      \begin{tabular}{rrrl}
        \toprule
        \multicolumn{2}{r}{Parameter} & \multicolumn{2}{l}{Value}\tabularnewline
        \midrule
        Discount factor & $\rho$ & 10\% & per annum\tabularnewline
        Interest rate & $r$ & 7\% & per annum\tabularnewline
        Drift & $\mu$ & 11\% & per annum\tabularnewline
        Volatility & $\xi$ & 30\% & per annum\tabularnewline
        Expiry & $T$ & 40 & years\tabularnewline
        Relative risk aversion & $1-\gamma$ & 0.7 & \tabularnewline
        Scaled transaction cost & $\lambda$ & 0.1 & \tabularnewline
        Fixed transaction cost & $C$ & 0.05 & \tabularnewline
        Maximum withdrawal rate & $w_{\text{max}}$ & 100 & \tabularnewline
        Initial stock value & $s_{0}$ & \$45.20 & \tabularnewline
        Initial bank account value & $b_{0}$ & \$45.20 & \tabularnewline
        \bottomrule
    \end{tabular}}
    \par
  \end{centering}
  \caption{Optimal consumption: parameters from
  \cite{chancelier2002combined}\label{tab:optimal_consumption_parameters}}
\end{table}

\subsubsection{Convergence of the direct control scheme}

As in $\S$\ref{subsec:exchange_rate_direct_control}, the domain
$[0,T]\times[0,\infty)^{2}$ and $Z(t,x)$ are truncated so that the
state after an impulse $x_{i}+\Gamma(t,x_{i},z_{i})$ remains in the
truncated domain. We use the notation $x_{i}=(s_{i},b_{i})$. The
direct control problem is given by \eqref{eq:problem_combined} subject
to \eqref{eq:control_set_for_numerical_schemes} and
\eqref{eq:direct_control_system}.

Suppose there exists a grid node $x_{i^{1}}$ and $P\coloneqq(w,z,\psi)$
such that $\psi_{i^{1}}=1$ and that there exists no path in $B(z)$
from $i^{1}$ to some $j$ with $\psi_{j}=0$. Since $C>0$, there
exists a path $i^{1}\rightarrow i^{2}\rightarrow\cdots$ of infinite
length such that $s_{i^{1}}+b_{i^{1}}>s_{i^{2}}+b_{i^{2}}>\cdots$
and $\psi_{i^{q}}=1$ for all $q$. Due to the finitude of the grid,
$x_{i^{q}}=x_{i^{\ell}}$ (and hence
$s_{i^{q}}+b_{i^{q}}=s_{i^{\ell}}+b_{i^{\ell}}$)
for some $q<\ell$, a contradiction. It follows that no such $x_{i^{1}}$
exists: \ref{ass:path_to_sdd} is satisfied.

\subsubsection{Optimal control}

As in \cite{chancelier2002combined}, three regions are observed in
an optimal control: the \emph{buy} (B), \emph{sell} (S), and
\emph{continuation/no
transaction }(NT) regions. In the B and S regions, the controller
intervenes by jumping back to the closest of the two lines marked
$\Delta_{1}$ and $\Delta_{2}$. In NT, the controller consumes continuously.
\begin{figure}
  \begin{centering}
    \subfloat[Optimal control]{
      \begin{centering}
        \includegraphics[height=2in]{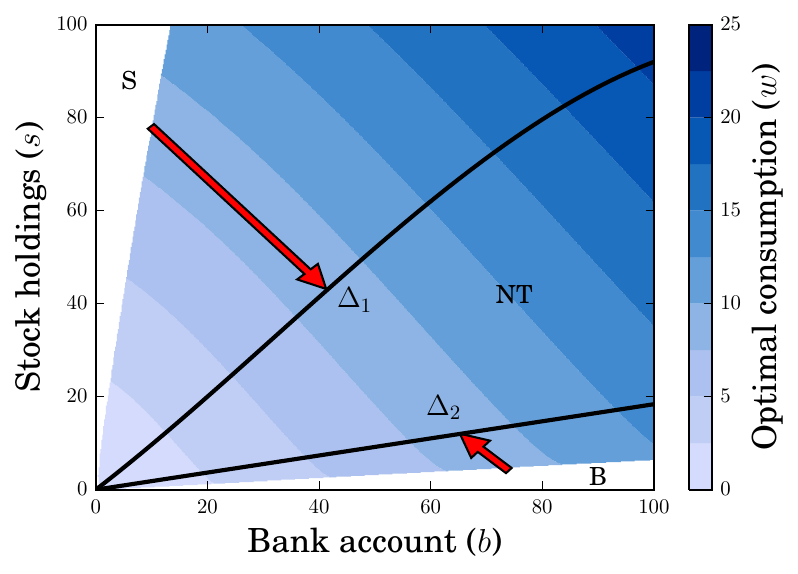}
        \par
      \end{centering}
    }\subfloat[Value]{
      \begin{centering}
        \includegraphics[height=2.3in]{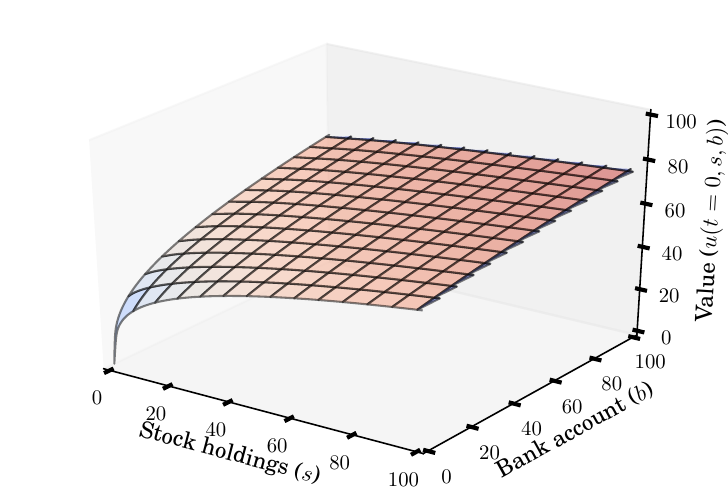}
        \par
      \end{centering}
    }
    \par
  \end{centering}
  \caption{Optimal consumption at initial time (compare with
    \cite[Figures 1 and 2]{chancelier2002combined})
  \label{fig:optimal_consumption_chancelier}}
\end{figure}

\subsubsection{Convergence tests}

Convergence tests are shown in \prettyref{tab:optimal_consumption_convergence}.
We mention that artificial boundary conditions
$\partial^{q}u/\partial s^{q}=0$
and $\partial u/\partial b=0$ are used at the truncated boundaries
$s=s_{\text{max}}$ and $b=b_{\text{max}}$. The results for the direct
control and penalized schemes are near-identical, though the former
requires significantly more policy iterations per timestep. The rate
of convergence for the semi-Lagrangian scheme becomes sublinear for
higher levels of refinement.
\begin{table}
  \begin{centering}
    {\footnotesize%
      \begin{tabular}{cccccc}
        \toprule
        $h$ & $s$ nodes & $b$ nodes & $w$ nodes & $z$ nodes &
        Timesteps\tabularnewline
        \midrule
        1 & 20 & 20 & 15 & 15 & 32\tabularnewline
        1/2 & 40 & 40 & 30 & 30 & 64\tabularnewline
        $\vdots$ & $\vdots$ & $\vdots$ & $\vdots$ & $\vdots$ &
        $\vdots$\tabularnewline
        \bottomrule
    \end{tabular}}
    \par
  \end{centering}
  \caption{Optimal consumption: numerical grid}
\end{table}
\begin{table}
  \begin{centering}
    \subfloat[Direct control]{
      \begin{centering}
        {\footnotesize%
          \begin{tabular}{cccccc}
            \toprule
            $h$ & $u(t=0,s_{0},b_{0})$ & Avg. policy its. & Avg.
            BiCGSTAB its. & Ratio & Norm. time\tabularnewline
            \midrule
            1 & 56.062123 & 7.63 & 1.28 &  & 1.63e+01\tabularnewline
            1/2 & 58.739224 & 8.80 & 1.90 &  & 2.93e+02\tabularnewline
            1/4 & 59.420125 & 10.4 & 2.28 & 3.93 & 5.66e+03\tabularnewline
            1/8 & 59.658413 & 11.8 & 3.28 & 2.86 & 1.03e+05\tabularnewline
            1/16 & 59.754780 & 13.3 & 4.45 & 2.47 & 1.85e+06\tabularnewline
            1/32 & 59.797206 & 14.2 & 6.54 & 2.27 & 3.05e+07\tabularnewline
            \bottomrule
        \end{tabular}}
        \par
      \end{centering}
    }
    \par
  \end{centering}
  \begin{centering}
    \subfloat[Penalized]{
      \begin{centering}
        {\footnotesize%
          \begin{tabular}{cccccc}
            \toprule
            $h$ & $u(t=0,s_{0},b_{0})$ & Avg. policy its. & Avg.
            BiCGSTAB its. & Ratio & Norm. time\tabularnewline
            \midrule
            1 & 56.058496 & 4.09 & 1.43 &  & 1.03e+01\tabularnewline
            1/2 & 58.739041 & 3.95 & 1.77 &  & 1.47e+02\tabularnewline
            1/4 & 59.420075 & 3.40 & 2.35 & 3.94 & 1.99e+03\tabularnewline
            1/8 & 59.658399 & 3.04 & 3.69 & 2.86 & 2.69e+04\tabularnewline
            1/16 & 59.754778 & 2.80 & 5.50 & 2.47 & 4.02e+05\tabularnewline
            1/32 & 59.797215 & 2.58 & 5.98 & 2.27 & 5.86e+06\tabularnewline
            \bottomrule
        \end{tabular}}
        \par
      \end{centering}
    }
    \par
  \end{centering}
  \begin{centering}
    \subfloat[Semi-Lagrangian]{
      \begin{centering}
        {\footnotesize%
          \begin{tabular}{ccccc}
            \toprule
            $h$ & $u(t=0,s_{0},b_{0})$ & Avg. BiCGSTAB its. & Ratio &
            Norm. time\tabularnewline
            \midrule
            1 & 55.621632 & 1.00 &  & 1.00e+00\tabularnewline
            1/2 & 58.782064 & 2.00 &  & 1.55e+01\tabularnewline
            1/4 & 59.404576 & 3.00 & 5.08 & 2.60e+02\tabularnewline
            1/8 & 59.569370 & 4.00 & 3.78 & 4.05e+03\tabularnewline
            1/16 & 59.651186 & 6.00 & 2.01 & 6.68e+04\tabularnewline
            1/32 & 59.705315 & 8.00 & 1.51 & 1.11e+06\tabularnewline
            1/64 & 59.748325 & 10.8 & 1.26 & 1.85e+07\tabularnewline
            \bottomrule
        \end{tabular}}
        \par
      \end{centering}
    }
    \par
  \end{centering}
  \caption{Optimal consumption: convergence
  tests\label{tab:optimal_consumption_convergence}}
\end{table}

\subsection{Guaranteed minimum withdrawal benefit (GMWB) in variable
annuities\label{subsec:gmwb}}

Guaranteed minimum withdrawal benefits (GMWB) in variable annuities
provide investors with the tax-deferred nature of variable annuities
along with a guaranteed minimum payment. GMWB pricing has been previously
considered as a singular control problem in
\cite{milevsky2006financial,dai2008guaranteed}
and as an impulse control problem in \cite{chen2008numerical}.
Optimal controls for GMWBs with annual withdrawals are considered in
\cite{azimzadeh2014}.

A GMWB is composed of investment and guarantee accounts, $(S_{t})_{t\geq0}$
and $(A_{t})_{t\geq0}$, respectively. It is bootstrapped via a lump
sum payment $s_{0}$ to an insurer, placed in the (risky) investment
account (i.e. $S_{0}=s_{0}$). A GMWB promises to pay back at least
the lump sum $s_{0}$, assuming that the holder of the contract does
not withdraw above a certain rate. This is captured by setting $A_{0}=s_{0}$
and reducing both investment and guarantee accounts on a dollar-for-dollar
basis upon withdrawals. The holder can continue to withdraw as long
as the guarantee account remains positive. In particular, at any point
in time until the expiry of the contract $T$, the holder may:
\begin{itemize}

  \item withdraw continuously at a rate of $G\geq0$ per annum regardless
    of the performance of the investment\textbf{ (stochastic control)};

  \item withdraw a finite amount $z$ instantaneously reduced by the
    excess withdrawal rate $0\leq\kappa\leq1$ \textbf{(impulse control)}.

\end{itemize} The holder gets the larger of the investment account
and a full withdrawal at expiry.

The guarantee account can be withdrawn from continuously or instantaneously:
\begin{align*}
  dA_{t} & =-w_{t}dt & \text{if }\tau_{j}<t<\tau_{j+1} &  &
  \text{(stochastic control)};\\
  A_{\tau_{j+1}} & =A_{\tau_{j+1}-}-z_{\tau_{j+1}} &  &  &
  \text{(impulse control)}.
\end{align*}
Let $\rho\geq0$ denote the risk-free rate. Consider an index $(Y_{t})_{t\geq0}$
following
\[
  dY_{t}=\rho Y_{t}dt+\xi Y_{t}d\mathfrak{W}_{t}
\]
under the risk-neutral measure. The investment account tracks the
index and is adjusted by withdrawals from the guarantee account:
\begin{align*}
  dS_{t} &
  =\left(\left(\rho-\eta\right)S_{t}-w_{t}\mathds{1}_{\{S_{t}>0\}}\right)dt+\xi
  S_{t}d\mathfrak{W}_{t} & \text{if }\tau_{j}<t<\tau_{j+1};\\
  S_{\tau_{j+1}} & =\max\left(S_{\tau_{j+1}-}-z_{\tau_{j+1}},0\right).
\end{align*}
$0\leq\eta\leq\rho$ is the proportional rate deducted from the investment
account and serves as a premium for the guarantee. A combined control
$\theta\coloneqq(w,\tau_{1},\tau_{2},\ldots,z_{1},z_{2},\ldots)$
is admissible if at all times, the guarantee account is nonnegative.
Let $\Theta$ denote the set of all admissible controls.

The insurer's worst-case cost of hedging (discussed in
\cite{azimzadeh2014hedging})
a GMWB at time $t$ with amount $S_{t}=s$ in the risky account and
amount $A_{t}=a$ is
\begin{multline*}
  u(t,s,a)\coloneqq e^{\rho
  t}\sup_{\theta\in\Theta}\mathbb{E}^{(t,s,a)}\biggl[\int_{t}^{T}e^{-\rho
    t^{\prime}}w_{t^{\prime}}dt^{\prime}+e^{-\rho
    T}\max\left(S_{T},\left(1-\kappa\right)A_{T}-C\right)\\
    +\sum_{\tau_{j}\leq
  T}e^{-\rho\tau_{j}}\left(\left(1-\kappa\right)z_{\tau_{j}}-C\right)\biggr]
\end{multline*}
where $C>0$ is a fixed transaction cost. The terminal payoff corresponds
to the maximum of the investment account or withdrawing the entirety
of the guarantee account at the excess withdrawal rate.

Let $x\coloneqq(s,a)$ and $\zeta\coloneqq\rho-\eta$. The associated
HJBQVI on $\Omega\coloneqq(0,\infty)^{2}$ and $\Lambda\coloneqq\emptyset$
is given by \eqref{eq:hjbqvi_pde_and_boundary} with
$g(T,x)\coloneqq\max(s,(1-\kappa)a-C)$
and
\begin{table}
  \begin{centering}
    {\footnotesize%
      \begin{tabular}{rrrl}
        \toprule
        \multicolumn{2}{r}{Parameter} & \multicolumn{2}{l}{Value}\tabularnewline
        \midrule
        Risk-free rate & $\rho$ & 5\% & per annum\tabularnewline
        Premium & $\eta$ & 0\% & per annum\tabularnewline
        Volatility & $\xi$ & 30\% & per annum\tabularnewline
        Expiry & $T$ & 10 & years\tabularnewline
        Withdrawal rate & $G$ & \$10 & per annum\tabularnewline
        Excess withdrawal rate & $\kappa$ & 10\% & \tabularnewline
        Fixed transaction cost & $C$ & $1/10^{6}$ & \tabularnewline
        Initial lump sum payment & $s_{0}$ & \$100 & \tabularnewline
        \bottomrule
    \end{tabular}}
    \par
  \end{centering}
  \caption{GMWB: parameters from
  \cite{chen2008numerical}\label{tab:gmwb_parameters}}
\end{table}
\begin{alignat*}{2}
  W & \coloneqq\left[0,G\right]; & \,Z(t,x) & \coloneqq\left[0,a\right];\\
  \text{\L}^{w} &
  \coloneqq\frac{\xi^{2}s^{2}}{2}\frac{\partial^{2}}{\partial
  s^{2}}+\zeta s\frac{\partial}{\partial s}+
  \begin{cases}
    -w\left(\frac{\partial}{\partial a}+\frac{\partial}{\partial
    s}\right) & \text{if }s,a>0;\\
    -w\frac{\partial}{\partial a} & \text{if }a>0;\\
    0 & \text{otherwise;}
  \end{cases} & \,\Gamma(t,x,z) &
  \coloneqq-\left(\min\left(z,s\right),z\right);\\
  f^{w} & \coloneqq
  \begin{cases}
    w & \text{if }a>0;\\
    0 & \text{otherwise;}
  \end{cases} & \,K(t,x,z) & \coloneqq\left(1-\kappa\right)z-C.
\end{alignat*}

\subsubsection{Convergence of the direct control scheme}

We use the notation $x_{i}=(s_{i},a_{i})$ and assume the origin $(0,0)$
is part of the numerical grid. The direct control problem is given
by \eqref{eq:problem_combined} subject to
\eqref{eq:control_set_for_numerical_schemes}
and \eqref{eq:direct_control_system}.

Suppose \ref{ass:subidempotence} is not satisfied so that for some
solution $v$, there exists $i$ such that
$v_{i}=[\mathbb{B}v]_{i}=[\mathbb{B}^{2}v]_{i}=\cdots$.
Since $C>0$, it follows that $v_{i}=-\infty$, a contradiction. Hence,
\ref{ass:subidempotence} holds.

We perform policy iteration on a modified problem with control set
$\mathcal{P}^{\prime}$ consisting of all controls $P\coloneqq(w,z,\psi)$
in $\mathcal{P}$ satisfying
\[
  \psi_{i}=0\text{ whenever }a_{i}=0\text{ and }z_{i}\neq0\text{
  }\text{whenever }a_{i}\neq0.
\]
As in \prettyref{exa:unidirectional}, \ref{ass:path_to_sdd}$^{\prime}$
follows from the unidirectionality of $z_{i}$.
\eqref{eq:modified_solves_original}
is established by noting that $z_{i}=0$ incurs an infinite cost (and
is therefore suboptimal). Convergence then follows from an application
of \prettyref{thm:the_standard_machinery}

\begin{rem} The condition $z_{i}\neq0$ appeals to intuition: the
  holder should never pay $C>0$ for a withdrawal of zero dollars.
\end{rem}

\subsubsection{Optimal control}

\prettyref{fig:gmwb_control} shows an optimal control for a GMWB,
corresponding to a worst-case cost of hedging from the perspective
of the insurer (optimality from the holder's perspective, who may
  have to take into
  consideration consumption, taxation, etc., is explored in
\cite{azimzadeh2014hedging}). We refer to \cite{chen2008numerical}
for an explanation of the
three distinct withdrawal regions.
\begin{figure}
  \begin{centering}
    \includegraphics[height=2.5in]{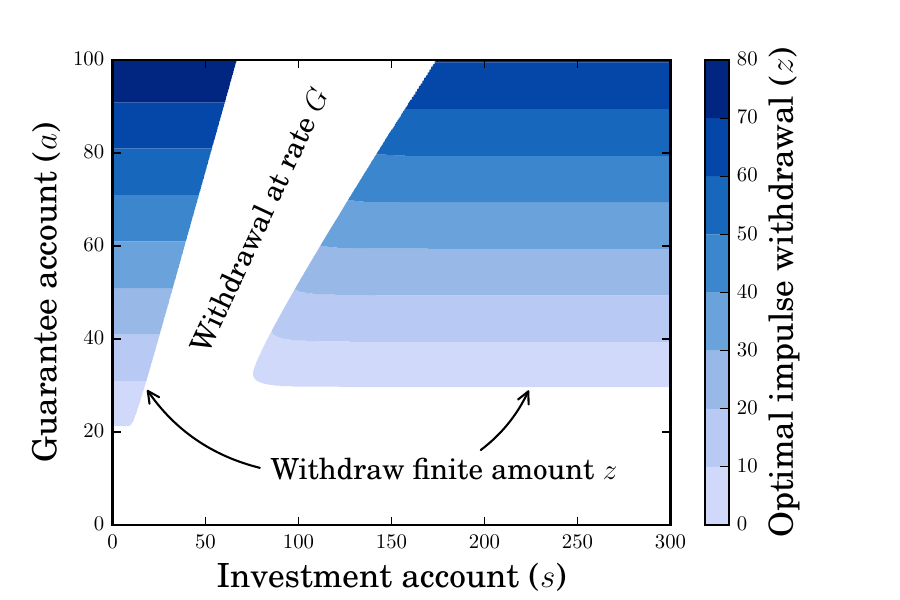}
    \par
  \end{centering}
  \caption{GMWB: optimal control at initial time with $\eta=0.03126$
  from \cite{chen2008numerical}\label{fig:gmwb_control}}
\end{figure}

\subsubsection{Convergence tests}

Convergence tests are shown in \prettyref{tab:gmwb_convergence}.
Since $w\mapsto\text{\L}(t,x,w)$ is linear, we take $W_{h}=\{0,G\}$
independent of $h$. An asymptotic boundary condition is used at the
truncated boundary $s=s_{\text{max}}$ (no boundary condition is needed
  at $a=a_{\text{max}}$ since the characteristics are outgoing in the
$a$ direction). For details, see \cite{chen2008numerical}. The direct
control and penalized schemes produce near-identical results and exhibit
similar execution times.
\begin{table}
  \begin{centering}
    {\footnotesize%
      \begin{tabular}{ccccc}
        \toprule
        $h$ & $s$ nodes & $a$ nodes & $z$ nodes & Timesteps\tabularnewline
        \midrule
        1 & 64 & 50 & 2 & 32\tabularnewline
        1/2 & 128 & 100 & 4 & 64\tabularnewline
        $\vdots$ & $\vdots$ & $\vdots$ & $\vdots$ & $\vdots$\tabularnewline
        \bottomrule
    \end{tabular}}
    \par
  \end{centering}
  \caption{GMWB: numerical grid}
\end{table}
\begin{table}
  \begin{centering}
    \subfloat[Direct control]{
      \begin{centering}
        {\footnotesize%
          \begin{tabular}{cccccc}
            \toprule
            $h$ & $u(t=0,s_{0},s_{0})$ & Avg. policy its. & Avg.
            BiCGSTAB its. & Ratio & Norm. time\tabularnewline
            \midrule
            1 & 107.68342 & 3.47 & 1.47 &  & 1.71e+01\tabularnewline
            1/2 & 107.70679 & 4.25 & 1.64 &  & 2.03e+02\tabularnewline
            1/4 & 107.71878 & 4.34 & 1.85 & 1.95 & 2.60e+03\tabularnewline
            1/8 & 107.72578 & 4.43 & 2.22 & 1.71 & 3.46e+04\tabularnewline
            1/16 & 107.72964 & 4.31 & 2.71 & 1.81 & 4.75e+05\tabularnewline
            1/32 & 107.73176 & 4.15 & 3.40 & 1.83 & 7.55e+06\tabularnewline
            \bottomrule
        \end{tabular}}
        \par
      \end{centering}
    }
    \par
  \end{centering}
  \begin{centering}
    \subfloat[Penalized]{
      \begin{centering}
        {\footnotesize%
          \begin{tabular}{cccccc}
            \toprule
            $h$ & $u(t=0,s_{0},s_{0})$ & Avg. policy its. & Avg.
            BiCGSTAB its. & Ratio & Norm. time\tabularnewline
            \midrule
            1 & 107.68243 & 3.47 & 1.58 &  & 1.80e+01\tabularnewline
            1/2 & 107.70639 & 4.08 & 1.65 &  & 2.06e+02\tabularnewline
            1/4 & 107.71870 & 3.95 & 1.76 & 1.95 & 2.45e+03\tabularnewline
            1/8 & 107.72576 & 3.98 & 1.97 & 1.74 & 3.22e+04\tabularnewline
            1/16 & 107.72964 & 3.71 & 2.39 & 1.82 & 4.34e+05\tabularnewline
            1/32 & 107.73175 & 3.32 & 3.01 & 1.83 & 6.62e+06\tabularnewline
            \bottomrule
        \end{tabular}}
        \par
      \end{centering}
    }
    \par
  \end{centering}
  \begin{centering}
    \subfloat[Semi-Lagrangian]{
      \begin{centering}
        {\footnotesize%
          \begin{tabular}{ccccc}
            \toprule
            $h$ & $u(t=0,s_{0},s_{0})$ & Avg. BiCGSTAB its. & Ratio &
            Norm. time\tabularnewline
            \midrule
            1 & 107.42351 & 1.00 &  & 1.00e+00\tabularnewline
            1/2 & 107.68443 & 1.00 &  & 1.02e+01\tabularnewline
            1/4 & 107.70841 & 1.00 & 10.9 & 1.39e+02\tabularnewline
            1/8 & 107.72257 & 1.00 & 1.70 & 2.03e+03\tabularnewline
            1/16 & 107.73015 & 1.00 & 1.87 & 3.31e+04\tabularnewline
            1/32 & 107.73224 & 1.98 & 3.62 & 6.10e+05\tabularnewline
            1/64 & 107.73337 & 2.90 & 1.85 & 1.13e+07\tabularnewline
            \bottomrule
        \end{tabular}}
        \par
      \end{centering}
    }
    \par
  \end{centering}
  \caption{GMWB: convergence tests\label{tab:gmwb_convergence}}
\end{table}

\section{Concluding remarks}

This work establishes the well-posedness of \eqref{eq:fixed_point_problem}
and gives sufficient conditions for convergence of the corresponding
policy iteration. \eqref{eq:fixed_point_problem} has applications
to the numerical solutions of HJBQVIs
($\S$\ref{sec:numerical_schemes}\textendash \ref{sec:examples})
and infinite-horizon MDPs with vanishing discount factor
(\prettyref{cor:optimal_control_of_markov_chain_with_vanishing_discount}).

A semi-Lagrangian scheme for the HJBQVI \eqref{eq:hjbqvi_pde_and_boundary}
is both easy to implement and requires only one linear solve per timestep.
However, it cannot be used if the diffusion or jump arrival rate of
the underlying stochastic process are control-dependent.

The direct control and penalized schemes do not suffer these limitations.
Numerical evidence suggests that both schemes perform similarly. However,
policy iteration applied to the direct control scheme can fail
(\prettyref{exa:policy_iteration_fails})
unless additional care is taken to remove certain suboptimal controls.
The removal of these controls is ad hoc (i.e. problem dependent).
Therefore, we recommend discretizing the problem with a penalized
scheme, applying policy iteration to solve the resulting nonlinear
equations.

\appendix

\section{General well-posedness of the Bellman problem
\eqref{eq:bellman_problem}\label{app:sequential_policy_iteration}}

By modifying policy iteration, it is possible to arrive at a version
of \prettyref{prop:convergence_of_policy_iteration} independent of
\assargsup. We can interpret this algorithm as taking into account
the error from approximating the supremum in $\proc{Policy-Iteration}$.
The algorithm, closely related to \cite[Algorithm Ho-4]{bokanowski2009some},
is given below (subject to the convention that for $x$ in $\mathbb{R}^{M}$
  and $c$ in $\mathbb{R}$, $x+c$ is the vector $x$ with $c$ added
to each component):

\begin{codebox}

  \Procname{\proc{$\epsilon$-Policy-Iteration}$(\mathcal{P},A(\cdot),b(\cdot),v^{0})$}

  \li    Pick a positive sequence $(\epsilon^{\ell})_{\ell\geq1}$
  in $\mathbb{R}$ such that $\sum_{\ell\geq1}\epsilon^{\ell}<\infty$

  \li \For $\ell\gets1,2,\ldots$ \li \Do

  Pick $P^{\ell}$ such that
  $-A(P^{\ell})v^{\ell-1}+b(P^{\ell})+\epsilon^{\ell}\geq\sup_{P\in\mathcal{P}}\{-A(P)v^{\ell-1}+b(P)\}$

  \li    $v^{\ell}\coloneqq\proc{Solve}(A(P^{\ell}),b(P^{\ell}),v^{\ell-1})$

  \End

\end{codebox}

The following appears in \cite{bokanowski2009some}:
\begin{lem}
  \label{lem:nearly_monotone_convergence}A bounded sequence
  $(v^{\ell})_{\ell\geq0}$
  in $\mathbb{R}$ converges if there exists a positive sequence
  $(\epsilon^{\ell})_{\ell\geq1}$
  in $\mathbb{R}$ such that $\sum_{\ell\geq1}\epsilon^{\ell}<\infty$
  and $v^{\ell}-v^{\ell-1}\geq-\epsilon^{\ell}$ for $\ell\geq1$.
\end{lem}

We require the following lemma, whose proof is trivial and thus omitted:
\begin{lem}
  \label{lem:marginalized_supremum} Let $X$ be a nonempty set, $Y$ a normed
  linear space, $T\colon X\times Y\rightarrow\mathbb{R}$, and
  $Q\colon X\rightarrow\mathbb{R}$
  with $Q$ bounded above. Suppose that for each $x$ in $X$,
  $T_{x}\colon Y\rightarrow\mathbb{R}$
  defined by $T_{x}(y)\coloneqq T(x,y)$ is linear and that $T_{x}$
  has operator norm bounded uniformly with respect to $x$. The map
  $y\mapsto\sup_{x\in X}\{T(x,y)+Q(x)\}$ is uniformly continuous.
\end{lem}

\begin{thm}
  \label{thm:convergence_of_sequential_policy_iteration} Suppose
  \ref{ass:inverse_is_bounded}, \assweaker, and that $A(P)$ is a
  monotone matrix for all $P$ in $\mathcal{P}$. $(v^{\ell})_{\ell\geq1}$
  defined by \proc{$\epsilon$-Policy-Iteration} converges to the unique
  solution $v$ of \eqref{eq:bellman_problem}.
\end{thm}
\begin{proof}
  First, note that
  \begin{equation}
    A(P^{\ell})\left(v^{\ell}-v^{\ell-1}\right)=-A(P^{\ell})v^{\ell-1}+b(P^{\ell})\geq\sup_{P\in\mathcal{P}}\left\{
    -A(P)v^{\ell-1}+b(P)\right\}
    -\epsilon^{\ell}.\label{eq:policy_iteration_difference}
  \end{equation}
  For $\ell>1$,
  \[
    \sup_{P\in\mathcal{P}}\left\{ -A(P)v^{\ell-1}+b(P)\right\}
    \geq-A(P^{\ell-1})v^{\ell-1}+b(P^{\ell-1})=0.
  \]
  Combining this with \eqref{eq:policy_iteration_difference},
  \[
    v^{\ell}-v^{\ell-1}\geq-A(P^{\ell})^{-1}(\epsilon^{\ell},\ldots,\epsilon^{\ell})^{\intercal}\geq-C\epsilon^{\ell}\text{
    for some }C\geq0.
  \]
  The last inequality follows from the boundedness of $P\mapsto A(P)^{-1}$
  in \ref{ass:inverse_is_bounded}. By
  \prettyref{lem:nearly_monotone_convergence},
  $v^{\ell}\rightarrow v$ for some $v$ in $\mathbb{R}^{M}$. Taking
  limits in \eqref{eq:policy_iteration_difference} and applying
  \prettyref{lem:marginalized_supremum},
  \[
    0=\lim_{\ell\rightarrow\infty}\left(\sup_{P\in\mathcal{P}}\left\{
      -A(P)v^{\ell-1}+b(P)\right\}
    \right)=\sup_{P\in\mathcal{P}}\left\{ -A(P)v+b(P)\right\} .
  \]
  Hence, $v$ is a solution to \eqref{eq:bellman_problem}. Uniqueness
  is proven similarly to \prettyref{thm:uniqueness}.
\end{proof}

\section{Proof of \prettyref{lem:invariance}\label{app:proof_of_invariance}}
\begin{proof}
  We write $A_{\delta}$ and $b_{\delta}$ to stress dependence on $\delta$.
  Let $v$ be a solution of \eqref{eq:problem_combined} with $\delta=1$.
  A pigeonhole principle argument allows us to pick a sequence
  $(P^{\ell})_{\ell\geq0}\coloneqq(w^{\ell},z^{\ell},\psi^{\ell})_{\ell\geq0}$
  in $\mathcal{P}$ such that $\psi^{\ell}=\psi$ is constant and
  $-A_{1}(P^{\ell})v+b_{1}(P^{\ell})\rightarrow0$.
  Multiplying both sides by $I-\Psi+\delta_{0}\Psi$ (where
  $\Psi\coloneqq\operatorname{diag}(\psi)$)
  yields $-A_{\delta_{0}}(P^{\ell})v+b_{\delta_{0}}(P^{\ell})\rightarrow0$,
  and hence
  $\sup_{P\in\mathcal{P}}\{-A_{\delta_{0}}(P)v+b_{\delta_{0}}(P)\}\geq0$.
  Supposing that this inequality is strict, it follows that for some
  $P \coloneqq (w, z, \psi^\prime)$ and $i$,
  $[-A_{\delta_{0}}(P)v+b_{\delta_{0}}(P)]_{i}>0$. Multiplying
  both sides by $1-\psi^\prime_i+\delta_0^{-1}\psi^\prime_i$ yields
  $[-A_{\delta=1}(P)v+b_{\delta=1}(P)]_{i}>0$,
  contradicting that $v$ is a solution. The converse is handled similarly.
\end{proof}

\bibliographystyle{plainnat}
\bibliography{main}

\end{document}